\documentclass[a4paper]{amsart}
\usepackage{mathpazo}
\usepackage[utf8]{inputenc}
\usepackage{amsmath, amssymb, amsthm, mathrsfs, mathtools, graphicx}
\usepackage{varioref}
\usepackage[backref=page]{hyperref}
\usepackage{enumitem, xspace, ifthen, comment}
\usepackage[all]{xy}
\usepackage[nameinlink, capitalize]{cleveref}
\usepackage[dvipsnames]{xcolor}
\usepackage{tikz}
\usepackage{tikz-cd}
\definecolor{violet}{rgb}{0.0,0.2,0.7}
\definecolor{rouge2}{rgb}{0.8,0.0,0.2}

\tikzset{
    labl/.style={anchor=south, rotate=90, inner sep=.5mm}
}

\setlist[enumerate]{
  label=(\thethm.\arabic*),
  itemsep=1ex
}

\setlist[itemize]{
  leftmargin=*,
  topsep=1ex,
  itemsep=1ex,
  label=$\circ$
}

\numberwithin{equation}{section}

\newtheorem*{thm-plain}{Theorem}
\newtheorem{thm}{Theorem}[section]
\newtheorem{lem}[thm]{Lemma}
\newtheorem{prp}[thm]{Proposition}
\newtheorem{cor}[thm]{Corollary}

\newtheorem{bigthm}{Theorem}

\newtheorem{aawlog}{Assumption}

\theoremstyle{definition}
\newtheorem{dfn}[thm]{Definition}

\newtheorem*{dfn-plain}{Definition}

\theoremstyle{remark}
\newtheorem{claim}[thm]{Claim}

\newtheorem{ntn}[thm]{Notation}
\newtheorem{setup}[thm]{Setup}
\newtheorem{rem}[thm]{Remark}

\newtheorem*{rem-plain}{Remark}

\newcommand{\inj}{\hookrightarrow}

\newcommand{\map}{\dashrightarrow}

\newcommand{\wb}{\overline}
\newcommand{\wh}{\widehat}

\newcommand{\factor}[2]{\left. \raise 2pt\hbox{$#1$} \right/\hskip -2pt \raise -2pt\hbox{$#2$}}

\newdir{ ir}{{}*!/-5pt/@^{(}} 
\newdir{ il}{{}*!/-5pt/@_{(}} 

\def\rd#1.{\lfloor{#1}\rfloor}
\def\rp#1.{\lceil{#1}\rceil}
\def\tw#1.{\langle{#1}\rangle}

\def\Hnought#1.#2.{\mathit{\Gamma} \!\left( #1, #2 \right)}
\def\HH#1.#2.#3.{\mathrm{H}^{#1} \!\left( #2, #3 \right)}
\def\hh#1.#2.#3.{h^{#1} \!\left( #2, #3 \right)}
\def\RR#1.#2.#3.{R^{#1} #2_* #3}
\def\HHc#1.#2.#3.{\mathrm{H}_{\mathrm{c}}^{#1} \!\left( #2, #3 \right)}
\def\Hh#1.#2.#3.{\mathrm{H}_{#1} \!\left( #2, #3 \right)}
\def\Hom#1.#2.{\mathrm{Hom} \!\left( #1, #2 \right)}
\def\End#1.{\mathrm{End} \!\left( #1 \right)}
\def\sHom#1.#2.{\mathscr{H}\!om \!\left( #1, #2 \right)}
\def\Ext#1.#2.#3.{\mathrm{Ext}^{#1} \!\left( #2, #3 \right)}
\def\sExt#1.#2.#3.{\mathscr{E}\!xt^{#1} \!\left( #2, #3 \right)}

\DeclareMathOperator{\Ric}{Ric}
\DeclareMathOperator{\tr}{tr}

\newcommand{\ep}{\varepsilon}
\renewcommand{\theta}{\vartheta}
\renewcommand{\phi}{\varphi}
\newcommand{\vp}{\varphi}
\newcommand{\vpte}{\varphi_{t,\ep}}
\newcommand{\om}{\omega}
\newcommand{\omke}{\omega_{\rm KE}}
\newcommand{\omkeg}{\omega_{\textrm{KE},\gamma}}

\newcommand{\omte}{\omega_{t, \ep}}
\newcommand{\X}{{\wh X} }

\newcommand{\Q}{\ensuremath{\mathbb Q}}
\newcommand{\R}{\ensuremath{\mathbb R}}

 \newcommand{\sE}{\mathscr E} \newcommand{\sF}{\mathscr F}
\newcommand{\sG}{\mathscr G} \newcommand{\sH}{\mathscr H} 
  
  \newcommand{\sO}{\mathscr O}

\newcommand{\sV}{\mathscr V}

  \newcommand{\cC}{\mathcal C}

\definecolor{forrest}{RGB}{81,133,49}
\definecolor{mydarkblue}{RGB}{10,92,153}

\title{A decomposition theorem for $\mathbb Q$-Fano Kähler-Einstein varieties}

\hypersetup{
  pdfauthor={Henri Guenancia},
  pdftitle={Stability of Q-Fano},
  pdfkeywords={},
  pdfstartview={Fit},
  pdfpagelayout={OneColumn},
  pdfpagemode={UseNone},
  linktoc=all,
  breaklinks,
  linkcolor=rouge2,
  citecolor=violet,
  urlcolor=[RGB]{0 96 0},
  colorlinks}

\author{Stéphane Druel}
\address{Univ Lyon, CNRS, Université Claude Bernard Lyon 1, UMR 5208, Institut Camille Jordan, F-69622 Villeurbanne, France}
\email{\href{mailto:stephane.druel@math.cnrs.fr}{stephane.druel@math.cnrs.fr}}
\urladdr{\href{http://druel.perso.math.cnrs.fr/index.html}{druel.perso.math.cnrs.fr}}

\author{Henri Guenancia}
\address{Institut de Math\'ematiques de Toulouse, Universit\'e Paul Sabatier, 31062 Toulouse Cedex 9, France}
\email{\href{mailto:henri.guenancia@math.cnrs.fr}{henri.guenancia@math.cnrs.fr}}
\urladdr{\href{https://hguenancia.perso.math.cnrs.fr/}{hguenancia.perso.math.cnrs.fr}}

\author{Mihai P\u{a}un}
\address{Lehrstuhl f\"ur Mathematik VIII, Universit\"at Bayreuth, 95440 Bayreuth, Germany}
\email{\href{mailto:mihai.paun@uni-bayreuth.de}{mihai.paun@uni-bayreuth.de}}

\begin{document}

\begin{abstract}
Let $X$ be a $\mathbb Q$-Fano variety admitting a Kähler-Einstein metric. We prove that up to a finite quasi-étale cover, $X$ splits isometrically as a product of Kähler-Einstein $\mathbb Q$-Fano varieties whose tangent sheaf is stable with respect to the anticanonical polarization. This relies among other things on a very general splitting theorem for algebraically integrable foliations. We also prove that the canonical extension of $T_X$ by $\sO_X$ is semistable with respect to the anticanonical polarization. 
\end{abstract}

\maketitle

 \begingroup
 \hypersetup{linkcolor=rouge2}
 \tableofcontents
 \endgroup

\section{Introduction}

Let $(X,\om)$ be a Fano Kähler-Einstein manifold, i.e. $X$ is a projective manifold with $-K_X$ ample and admitting a Kähler metric $\om$ solving $\Ric \om=\om$. It follows from the (easy direction of the) Kobayashi-Hitchin correspondence that the tangent bundle of $X$ splits as a direct sum of parallel subbundles
\begin{equation}
\label{TX}
T_X=\bigoplus_{i\in I} F_i
\end{equation}
such that $F_i$ is stable with respect to $-K_X$. Since $X$ is simply connected, de Rham's splitting theorem asserts that one can integrate the foliations arising in decomposition~\eqref{TX} and obtain an isometric splitting
$$(X,\om) \simeq \prod_{i\in I} (X_i,\om_i)$$
into Kähler-Einstein Fano manifolds which is compatible with \eqref{TX}.\\

Over the last few decades, a lot of attention has been drawn to projective varieties with mild singularities, in relation to the progress of the Minimal Model Program (MMP). In that context, the notion of $\Q$-Fano variety (cf. Definition~\ref{QFano}) has emerged and played a central role in birational geometry. 

On the analytic side, singular Kähler-Einstein metrics have been introduced and constructed in various settings (see e.g. \cite{EGZ, BBEGZ, BG} and Definition~\ref{defKE}). They induce genuine Kähler-Einstein metrics on the regular part of the variety but are in general incomplete, preventing the use of most useful results in differential geometry (like the de Rham's splitting theorem mentionned above) to analyze their behavior.  However, these objects are well-suited to study (poly)-stability properties of the tangent sheaf as it was observed by \cite{GSS}, relying on earlier results by \cite{Enoki}. 

In the Ricci-flat case, the holonomy of the singular metrics was computed in \cite{GGK}. Moreover, \cite{Dru16} provided an algebraic integrability result for foliations as well as a splitting result in that setting. Building upon those results, Höring and Peternell \cite{HP} could eventually prove the singular version of the Beauville-Bogomolov decomposition theorem.

In the positive curvature case, some new difficulties arise. In this paper, our main contribution is to single out and overcome those difficulties in order to prove the following structure theorem for $\Q$-Fano varieties that admit a Kähler-Einstein metric.

\begin{bigthm}\label{thm:main_thm}
Let $X$ be a $\mathbb Q$-Fano variety admitting a Kähler-Einstein metric $\om$. Then 
$T_X$ is polystable with respect to $c_1(X)$. Moreover, there exists a quasi-\'etale cover $f\colon Y \to X$ 
such that $(Y,f^*\omega)$ decomposes isometrically as a product $$(Y,f^*\om) \simeq \prod_{i\in I}(Y_i,\om_i),$$
where $Y_i$ is a $\mathbb Q$-Fano variety with stable tangent sheaf with respect to $c_1(Y_i)$ and $\omega_i$ is a Kähler-Einstein metric
on $Y_i$.
\end{bigthm}

It was proved very recently by Braun \cite[Theorem~2]{Braun20} that the fundamental group of the regular locus of a $\Q$-Fano variety is finite. Relying on that result, one can refine Theorem~\ref{thm:main_thm} and obtain that the varieties $Y_i$ satisfy the additional property: $\pi_1(Y_i^{\rm reg})=\{1\}$. \\

 Note that the semistability of $T_X$ for a Kähler-Einstein $\Q$-Fano variety $X$ was proved by Chi Li in \cite[Proposition~3.7]{ChiLi18} in the case where $X$ admits a resolution where all exceptional divisors have non-positive discrepancy, e.g. a crepant resolution. 

\subsection*{Strategy of proof of Theorem~\ref{thm:main_thm}} \,

\medskip 
There are two main steps in the proof of 
Theorem~\ref{thm:main_thm}.

\medskip 

$\bullet$ The first step is the object of Theorem~\ref{polystable} where one proves that $T_X$ is the direct sum of stable subsheaves that are parallel with respect to the Kähler-Einstein metric $\om$ on $X_{\rm reg}$. This is achieved by computing slopes of subsheaves using the metric induced by the Kähler-Einstein metric and using Griffiths' well-known formula for the curvature of a subbundle. However, the presence of singularities (for $X$ and $\omega$) makes it hard to carry out the analysis directly on $X$. One has to work on a resolution using approximate Kähler-Einstein metrics as in \cite{GSS}. Yet an additional error term appears in the Fano case, requiring to introduce some new ideas to deal with it. 

\medskip

$\bullet$ The main result of the second step is a very general splitting theorem for algebraically integrable foliations, cf. Theorem~\ref{thm:splitting}. 
The context is as follows: after the first step above is completed, we know that the tangent bundle of $X$
splits as direct sum of foliations, say $(\mathcal F_i)_{i\in I}$. Since each $\mathcal F_i$
admits a complement inside $T_X$, it is weakly regular. It turns out that weakly regular foliations have many nice properties. The important fact which is established here is that an algebraically integrable, weakly regular foliation on a $\Q$-factorial projective variety with klt singularities is induced by a surjective, equidimensional morphism $X\to Y$, cf. 
Theorem~\ref{thm:regular_foliation_morphism}. When combined with suitable generalisations of techniques and results in \cite{cd1fzerocan}, this leads to the proof of Theorem~\ref{thm:splitting}.

\medskip
Finally, Theorem~\ref{thm:main_thm} can be proved by applying the splitting theorem from the second step to the foliations induced by the Kähler-Einstein metric as showed in the first step. Note that the algebraic integrability of these foliations follows from the deep results of \cite{bogomolov_mcquillan01}.\\

Our second main result is the following generalisation of a theorem of Tian  \cite[Theorem~0.1]{tian92}, which is a way to express some "strong" semistability of $T_X$.

\begin{bigthm}\label{thm2}
Let $X$ be a $\mathbb Q$-Fano variety admitting a Kähler-Einstein metric. Then the canonical extension of $T_X$ by $\sO_X$ is semistable with respect to $c_1(X)$.
\end{bigthm}

We refer to Section~\ref{section: canonical extension} for the construction of the canonical extension. Note that one can slightly improve the result above by showing that the canonical extension is actually polystable, cf. Remark~\ref{rem:polystable}.

 \subsection*{Strategy of proof of Theorem~\ref{thm2}}\,
 
 \medskip
 
The proof of Theorem~\ref{thm2} takes up most of \S~\ref{semi}. 
It relies largely on the computations carried out in \S~\ref{poly} to prove the polystability of $T_X$, but on top of those, several new ideas are needed to overcome the presence of singularities. 
 
 First, one needs reduce the statement to one on a resolution in order to use analytic methods.  As we do that and introduce approximate Kähler-Einstein metrics, we need to change the vector bundle structure; this has the effect to make it impossible to compute directly the slope of a subsheaf of the canonical extension. We explain how to overcome this difficulty in the first step of the proof of Theorem~\ref{gang}. The rest of the proof uses a combination of the original idea of Tian and the computations of \S~\ref{poly}.\\

 \subsection*{Some remarks about Theorems~\ref{thm:main_thm}~\&~\ref{thm2}} \,
 
\medskip

$\bullet$ One can express Theorem~\ref{thm:main_thm} and Theorem~\ref{thm2} in a purely algebraic way using the notion of $K$-stability, cf. Remark~\ref{Kstable}. \\

$\bullet $ The conclusion of Theorem~\ref{thm2} remains true under the more general assumption that the greatest Ricci lower bound of $X$ is $1$, cf. Theorem~\ref{semistable}. \\

 $\bullet$ Combining the recent uniformization result  \cite[Theorem~1.3]{GKP20} with Theorem~\ref{thm2}, one can prove that a Kähler-Einstein $\Q$-Fano variety that achieves equality in the Miyaoka-Yau inequality is a quotient of the projective space, cf. \cite[Theorem~1.5]{GKP20} and the few lines above. Moreover, one can relax the Kähler-Einstein condition and only require that the greater Ricci lower bound is one, cf. item above. 
  
 \subsection*{Acknowledgements} It is our pleasure to thank Daniel Greb, Stefan Kebekus and Thomas Peternell for sharing their results \cite{GKP20} and encouraging us to prove Theorem~\ref{thm2}. H. G. thanks Sébastien Boucksom for useful discussions about Remark~\ref{Kstable}.
S. D. was partially supported by the ERC project ALKAGE (ERC grant Nr 670846), the CAPES-COFECUB project Ma932/19 and the ANR project Foliage (ANR grant Nr ANR-16-CE40-0008-01). H.G. was partially supported by the ANR project GRACK.

\section{Polystability of the tangent sheaf}
\label{poly}
\subsection{Set-up}
\label{set-up}
\subsubsection{Notation}

\begin{dfn}
\label{QFano}
Let $X$ be a projective variety of dimension $n$. We say that $X$ is a $\mathbb Q$-Fano variety if $X$ has klt singularities and $-K_X$ is an ample $\mathbb Q$-line bundle.
\end{dfn}

We also recall the definition of (twisted) singular Kähler-Einstein metric, cf. \cite{BBEGZ}. 

\begin{dfn}
\label{defKE}
Let $X$ be a $\mathbb Q$-Fano variety, let $\theta\in c_1(X)$ be a smooth representative and let $\gamma \in [0,1)$. A twisted Kähler-Einstein metric relatively to the couple $(\theta, \gamma)$ is a  closed, positive current $\omkeg \in c_1(X)$ with bounded potentials, which is smooth on $X_{\rm reg}$ and satisfies
$$\Ric \omkeg=(1-\gamma) \omkeg+\gamma \theta$$
on that open set. 

\noindent
When $\gamma=0$, we write $\omke:=\om_{{\rm KE}, 0}$ and we call it a Kähler-Einstein metric.
\end{dfn}

\begin{rem}
\label{criterion}
By \cite[Proposition~3.8]{BBEGZ}, a smooth Kähler metric $\omega \in c_1(X_{\rm reg})$ on $X_{\rm reg}$ satisfying $\Ric \om = \om$ extends to a Kähler-Einstein metric in the sense of Definition~\ref{defKE} if and only if $\int_{X_{\rm reg}} \om^n= c_1(X)^n$. In particular, if $f:Y\to X$ is a (finite) quasi-étale cover between $\mathbb Q$-Fano varieties and $\omke$ is a Kähler-Einstein metric on $X$, then $f^*\omke$ is a Kähler-Einstein metric on $Y$. 
\end{rem}

 Let $\omega_X\in c_1(X)$ be a fixed Kähler metric on $X$. We will systematically make either one of the following assumptions: 

\begin{aawlog}
\label{AAKE}
For any $\gamma \in (0,1)$ small enough, there exists a twisted Kähler-Einstein metric $\omkeg$ on $X$ relatively to $(\om_X, \gamma)$.
\end{aawlog}

\begin{aawlog}
\label{AKE}
There exists a Kähler-Einstein metric $\omke$ on $X$.
\end{aawlog}

\begin{rem}
\label{Kstable}
One can rephrase the Assumptions~\ref{AAKE}-\ref{AKE} using the algebraic notion of $K$-stability. It follows from \cite{LTW} (building upon results of \cite{CDS1,CDS2,CDS3}, \cite{T}, \cite{ChiLi17}, \cite{BBJ} in the smooth case)  that 
\begin{enumerate}
\item[$\bullet$] $X$ satisfies Assumption~\ref{AAKE} if and only if $X$ is $K$-semistable.
\item[$\bullet$] $X$ satisfies Assumption~\ref{AKE} if $X$ is uniformly $K$-stable, and the converse holds provided $\mathrm{Aut}^\circ(X)=\{1\}$.
\end{enumerate}
\end{rem}

\begin{ntn}
Let $\pi:\X\to X$ be a resolution of singularities of $X$ with exceptional divisor $E=\sum_{k\in I} E_k$ and discrepancies $a_k>-1$ given by
$$K_{\X}=\pi^*K_X +\sum a_k E_k.$$
There exist numbers $\ep_k\in \mathbb Q_+$ such that the cohomology class $\pi^*c_1(X)-\sum \ep_k c_1(E_k)$ contains a Kähler metric $\om_\X$. We fix it for the rest of the paper.  Next, we pick sections $s_k\in H^0(\X,\mathcal O_{\X}(E_k))$ such that $E_k=(s_k=0)$, smooth hermitian metric $h_k$ on $\mathcal O_{\X}(E_k)$ with Chern curvature $\theta_k:=i\Theta_{h_k}(E_k)$
and a volume form $dV$ on $\X$ such that $\Ric dV=\pi^*\om_X-\sum_{k\in I} a_k \theta_k$. We set 
\begin{equation}
\label{he}
h_E:=\prod_{k\in I} h_k
\end{equation}
which defines a smooth metric on $\sO_\X(E)$. \\
\end{ntn}

\subsubsection{The twisted Kähler-Einstein metric and its regularizations} 

\noindent
In this section, we assume that either Assumption~\ref{AAKE} or Assumption~\ref{AKE} is fulfilled so that there exists a (twisted) Kähler-Einstein metric $\omkeg$
\begin{itemize}
\item either for any $\gamma \in [0,1)$ such that $0<\gamma \ll 1$
\item or for $\gamma=0$. 
\end{itemize}
For the time being, the parameter $\gamma$ is \textit{fixed}.\\

We denote by $\pi^*\omkeg= \pi^*\om_X+dd^c \vp$ the singular metric solving
$$(\pi^*\om_X+dd^c \vp)^n=e^{-(1-\gamma)\vp}fdV$$
where $f=\prod_{i\in I} |s_i|^{2a_i} \in L^p(dV)$ for some $p>1$. It is known that $\vp$ is bounded (even continuous) on $\X$ and smooth outside $E$, cf. \cite{BBEGZ}. Note that $\varphi$ depends on $\gamma$, but as notation will get quite heavy later, we choose not to highlight that dependence.\\

Next, we choose a family $\psi_\ep\in \mathscr C^{\infty}(\X)$ of quasi-psh functions on $\X$ such that 
\begin{enumerate}
\item[$\bullet$] One has $\psi_\ep \to \vp$ in $L^1(\X)$ and in $\mathscr C^{\infty}_{\rm loc}(\X\setminus E)$. 
\item[$\bullet$] There exists $C>0$ such that $\|\psi_\ep\|_{L^{\infty}(\X) }\le C.$
\item[$\bullet$] There exists a continuous function $\kappa: [0,1]\to \mathbb R_+$ with $\kappa(0)=0$ such that $\pi^*\om_X+dd^c\psi_\ep \ge -\kappa(\ep) \om_{\X}$.
\end{enumerate}
This is a standard application of Demailly's regularization results. The smooth convergence outside $E$ claimed in the first item follows from the explicit expression of the function $\psi_{\ep}$, see e.g. \cite[(3.3)]{LDN}.\\

For $\ep, t \ge 0$, one introduces the unique function $\vpte \in L^{\infty}(X) \cap \mathrm{PSH}(\X, \pi^*\om_X+t \om_\X)$ solving
$$\begin{cases}
(\pi^*\om_X+t \om_\X+dd^c \vpte)^n = f_\ep e^{-(1-\gamma)\psi_\ep} e^{-c_t}dV\\
\sup_{\X} \vpte =0
\end{cases}$$
where 
\begin{enumerate}
\item[$\bullet$] $f_\ep:=e^{a_\ep}\prod (|s_i|^2+\ep^2)^{a_i}$,
\item[$\bullet$] $a_\ep$ is a normalizing constant such that $\int_\X f_\ep e^{-(1-\gamma)\psi_\ep}dV=c_1(X)^n$; it converges to $1$ when $\ep\to0$.
\item[$\bullet$] $c_t$ is defined by $\{\pi^*\om_X+t\om_\X\}^n=e^{c_t} \cdot c_1(X)^n$. 
\end{enumerate}
The existence and uniqueness of $\vpte$ follows from Yau's theorem \cite{Yau78} when $t,\ep>0$ (in which case $\vpte$ is actually smooth) while the general case is treated in \cite{EGZ}. It follows from \textit{ibid.} that there exists a constant $C>0$ such that 
\begin{equation}
\label{bounded}
\|\vpte\|_{L^{\infty}(X)} \le C
\end{equation}
for any $t,\ep \in [0,1]$. Moreover, any weak limit $\wh \varphi$ of a sequence $(\varphi_{t_k, \ep_k})$ is bounded and is a smooth limit outside $E$. Therefore, it solves the equation 
\[(\pi^*\om_X+dd^c \wh \vp)^n=e^{-(1-\gamma)\varphi}fdV \] 
on $\X$. By the uniqueness result \cite[Thm.~A]{EGZ}, we have $\wh \varphi=\varphi$. That is
\begin{equation}
\label{limit}
\vpte \underset{t,\ep \to 0}{\longrightarrow} \varphi \quad \mbox{in } L^1(\X) \,\, \mbox{and in } \mathscr C^{\infty}_{\rm loc}(\X\setminus E).
\end{equation}

\noindent
One sets 
\begin{equation}
\label{omte}
\omte:=\pi^*\om_X+t \om_\X+dd^c \vpte
\end{equation} which solves the equation
\begin{equation}
\label{ric}
\Ric \omte = \pi^*\om_X+(1-\gamma)dd^c \psi_\ep-\Theta_\ep
\end{equation}
where 
\begin{equation}
\label{thetaep}
\Theta_\ep=\Theta(E,h_E^{\ep})=\sum a_i \theta_{i,\ep}
\end{equation} is the curvature of 
\begin{equation}
\label{hep}
h_E^{\ep}=\prod_i (|s_i|^2+\ep^2)^{-1} h_i
\end{equation} and $\theta_{i,\ep}=\theta_i+dd^c \log(|s_i|^2+\ep^2)$ converges to the current of integration along $E_i$ when $\ep\to 0$. \\

\subsection{Stability of $T_X$.}
\label{stab}
Setup and notation as in \S~\ref{set-up}. 

Let $\sF\subset T_\X$ be a subsheaf of rank $r$. We can assume that $\sF$ is saturated in $T_\X$, i.e. $T_\X/\sF$ is torsion-free. This is because saturating a subsheaf increases its slope.

From now on, we choose small numbers $t,\ep>0$ which we will later let go to zero. The Kähler metric $\omte$ defined in \eqref{omte} induces an hermitian metric $h_{t,\ep}$ on $T_\X$ which in turn induces a hermitian metric $h_F$ on $F:=\sF|_W$ where $W\subset \X$ is the maximal locus where $\sF$ is a subbundle of $T_\X$. Then, it is classical (see e.g. \cite[Rem.~8.5]{Koba}) that one can compute the slope of $\sF$ by integrating the trace of the first Chern form of $(F,h_F)$ over $W$, i.e. 
\begin{equation}
\label{degre}
\int_{W} c_1(F,h_F) \wedge \omega_{t,\ep}^{n-1}= c_1(\sF) \cdot \{\omega_{t,\ep}\}^{n-1}.
\end{equation}
On $W$, we have the following standard identity (cf. e.g. \cite[Thm.~14.5]{Dem1})
\begin{equation*}
i\Theta(F,h_F) = \mathrm{pr}_F \big(i\Theta(T_\X,h_{t,\ep})|_F \big)+\beta_{t,\ep}\wedge \beta_{t,\ep}^*
\end{equation*}
where $\beta \in \mathcal C^{\infty}_{0,1}(W, \mathrm{Hom}(T_\X, F))$ and $\beta^*$ is its adjoint with respect to $h_{t,\ep}$ and $h_F$. Therefore, we get
{\small
\begin{equation}
\label{curvature2}
c_1(F,h_F)\wedge \omte^{n-1} = \mathrm{tr}_{\rm End}\left( \mathrm{pr}_F \big(i\Theta(T_\X,h_{t,\ep})|_F \big)\right) \wedge \omte^{n-1}+\mathrm{tr}_{\rm End} (\beta_{t,\ep}\wedge \beta_{t,\ep}^*\wedge \omte^{n-1})
\end{equation}}

By \eqref{degre}, the integral of the left-hand side over $W$, yields $r$ times the slope of $\sF$ with respect to $\{\pi^*\om_X+t\om_\X\}$. As for the right-hand side, one can simplify the first term using the formula 
\begin{equation}
\label{contraction}
n \cdot i\Theta(T_\X,h_{t,\ep}) \wedge \omte^{n-1}=(\sharp \Ric \omte) \, \omte^n.
\end{equation}
Here we denote by $\sharp \Ric \omte$ the endomorphism of $T_\X$ induced by the Ricci curvature of $\omte$.

The equation \eqref{ric} is equivalent to
\begin{equation}
\label{ricci}
\Ric \omte = (1-\gamma)\omte+\gamma \pi^*\om_X-t\om_\X+(1-\gamma)dd^c(\psi_\ep-\vpte)-\Theta_\ep.
\end{equation}
Using the formula above, one gets 
\begin{align*}
\mu_{\omte}(\sF) \le& (1-\gamma)\mu_{\omte}(T_X) + \frac{1-\gamma}{nr}\underbrace{\int_\X  \mathrm{tr}_{\rm End}\mathrm{pr}_F(\sharp dd^c(\psi_\ep-\vpte))|_F \, \omte^{n}}_{=:{\rm \bf(I)}}\\
&+ \underbrace{\frac{\gamma}{n r}\int_\X \mathrm{tr}_{\rm End} \mathrm{pr}_F(\sharp \pi^*\omega_X)|_F \, \omte^{n}}_{=:{\rm \bf (II)}}
 -\frac 1{nr} \underbrace{\int_\X \mathrm{tr}_{\rm End} \mathrm{pr}_F(\sharp \Theta_\ep)|_F \, \omte^{n}}_{=:{\rm \bf (III)}}\\
 &+\frac 1{nr}\underbrace{\int_{W}\mathrm{tr}_{\rm End} (\beta_{t,\ep}\wedge \beta_{t,\ep}^*\wedge \omte^{n-1})}_{=:{\rm \bf (IV)}}.
\end{align*}

We therefore have four terms to deal with. To deal with ${\rm \bf (II)}-{\rm \bf (IV)}$, we will use the same computations as in \cite{GSS}, cf. explanations below. The main new term is ${\rm \bf (I)}$, which we treat first.

\medskip

\textbf{ The term ${\rm \bf (I)}$.}

\medskip

\noindent
It arises from the fact that, say when $\gamma=1$, we can not necessarily solve the perturbed equation $\Ric \omte = \omte-t\om_\X-\Theta_\ep$ unlike in the case where $K_X$ is ample or trivial. If all the discrepancies $a_i$ were negative, one could likely still solve that equation  using e.g. properness of Ding functional but we will not expand on that.\\

In order to deal with ${\rm \bf (I)}$, one makes the following observations: \\

$\bullet$ Given $\delta >0$, there exist $\eta=\eta(\delta)>0$ and an open neighborhood $U_\delta$ of $E\subset \X$ such that
\begin{equation}
\label{CN}
\forall \ep,t \le \eta, \quad \int_{U_{\delta}} (\om_{\psi_\ep}+\omte) \wedge \omte^{n-1} \le \delta
\end{equation}
where $\om_{\psi_\ep}=\pi^*\om_X+t\om_\X+dd^c \psi_{\ep}$. This inequality is a consequence of the Chern-Levine-Nirenberg inequality along with the bound of the potentials below
\begin{equation}
\label{bound}
\exists C>0, \forall \ep, t,  \quad \|\vpte \|_{L^\infty(\X)}+\|\psi_\ep \|_{L^\infty(\X)} \le C
\end{equation}
that we infer from \eqref{bounded}. Indeed, as explained in \cite{GSS}
one proceeds as follows. Let $\big(\Xi_\delta\big)_{\delta> 0}$ be a family of functions defined on $\mathbb R_+$, such that $\Xi_\delta(x)= 0$ if $x\leq \delta^{-1}$ and $\Xi_\delta(x)= 1$ if $x\geq 1+ \delta^{-1}$. Moreover we can assume that the derivative
of $\Xi_\delta$ is bounded by a constant independent of $\delta$.
Then we evaluate the quantity
\begin{equation}
\int_{\X} \Xi_\delta\big(\log\log\frac{1}{|s_E|^2}\big)(\om_{\psi_\ep}+\omte) \wedge \omte^{n-1} 
\end{equation}
and the proof of the classical Chern-Levine-Nirenberg inequality shows that the integral in \eqref{CN}
is smaller than
\begin{equation}\label{Poinc}
\int_{U_{\delta}}\omega_{E}^n
\end{equation}  
up to a constant which is independent of $t, \ep$. In \eqref{Poinc} we denote
by $\omega_{E}$ a metric with Poincar\'e singularities along the divisor $E$,
and by $U_{\delta}$ the support of the truncation function $\displaystyle \Xi_\delta\big(\log\log\frac{1}{|s_E|^2}\big)$. Here the main point is that the norm of the 
Hessian of the truncation function is uniformly bounded when measured with respect to $\omega_{E}$. 
The conclusion follows.
\smallskip

The hermitian endormorphism $\sharp dd^c(\psi_\ep-\vpte)$ is dominated (in absolute value) by the positive endomorphism $$\sharp (\om_{\psi_\ep}+\omte)$$ whose endomorphism trace is nothing by $\mathrm{tr}_{\omte}(\om_{\psi_\ep}+\omte)$. By \eqref{CN}, we are done with ${\rm \bf (I)}$ on $U_\delta$. \\

$\bullet$ The second observation is that given $K\Subset \X\setminus E$, there exists $\eta=\eta(K)>0$ such that 
\begin{equation}
\label{Cinfty}
\forall \ep,t \le \eta, \quad \|\psi_\ep-\vpte\|_{\mathscr C^2(K)} \le \delta.
\end{equation}
This is a consequence of the fact that $(\vpte)$ and $(\psi_\ep)$ converge uniformly (in $\ep,t$) to $\vp$ on $K$ by stability of the Monge-Ampère operator, cf. e.g. \cite[Thm.~C]{GZ11}, and have uniformly bounded $\mathscr C^p(K)$ norm for any $p$ thanks to  \eqref{bound}, Tsuji's trick and Evans-Krylov plus Schauder estimates. 

Therefore, one has $\pm \sharp dd^c(\psi_\ep-\vpte) \le \delta \om_\X$ hence ${\rm \bf (I)}$  is controlled on $K$ by $\delta \int_K \om_\X\wedge \omte^n \le C\delta$. \\

\textit{Conclusion.} 
Let $F_{t,\ep}:=\mathrm{tr}_{\rm End}\mathrm{pr}_F(\sharp dd^c(\psi_\ep-\vpte))|_F \, \omte^{n}$. 
One fixes $\delta>0$. We get a neighborhood $U_\delta$ of $E$ and a number $\eta'=\eta'(\delta)>0$ such that $\int_{U_\delta} F_{t,\ep}\le \delta$ for any $\ep,t\le \eta'$. Applying the second observation to $K=\X\setminus U_\delta$, we find $\eta''=\eta''(\delta)$ such that $\int_{X\setminus U_\delta} F_{t,\ep} \le C \delta$ for any $\ep,t \le \eta''$. Choosing $\eta:=\min\{\eta',\eta''\}$, we find that 
$$\forall \ep, t \le \eta, \quad \int_\X F_{t,\ep} \le C'\delta.$$
In short, the term ${\rm \bf (I)}$ converges to zero when $\ep,t\to 0$. 

\medskip

\textbf{ The term ${\rm \bf (II)}$.}

\medskip

\noindent
As $\pi^*\om_X \ge0$, one has 
\begin{align*}
\mathrm{tr}_{\rm End} \mathrm{pr}_F(\sharp \pi^*\omega_X)|_F \omte^n & \le \mathrm{tr}_{\rm End} (\sharp \pi^*\omega_X) \omte^n \\
& =\tr_{\omte}(\pi^*\om_X) = n \,\pi^*\om_X \wedge \omte^{n-1}.
\end{align*}
Integrating over $X$, one finds 
\begin{equation*}
{\rm \bf (II)} \le \gamma  r^{-1} (\pi^*c_1(X) \cdot \{\omte\}^{n-1})
\end{equation*}
and the right-hand side converges to $\frac{\gamma n}{r} \mu(T_\X)$ when $t\to 0$, where the slope is taken with respect to $\pi^*c_1(X)$.

\medskip

\textbf{ The term ${\rm \bf (III)}$.}

\medskip

\noindent
As said above, the arguments to treat this term are borrowed from \cite{GSS}. For the convenience of the reader, we will recall the important steps. To lighten notation, we will drop the index $i$. One can write $\Theta_\ep=  \frac{\ep^2 lD's|^2}{(|s|^2+\ep^2)^2}+ \frac{\ep^2}{|s|^2+\ep^2} \cdot \theta$. Let us set $g_\ep:=  \frac{\ep^2}{|s|^2+\ep^2} $. Up to rescaling $\omega_\X$, one can assume that $-\om_\X\le \theta \le \om_\X$ so that $\Theta_\ep +g_\ep\om_\X \ge 0$. Then one sees easily that
\begin{align*}
 \mathrm{tr}_{\rm End} \mathrm{pr}_F(\sharp \Theta_\ep)|_F \, \omte^{n} & \le  \mathrm{tr}_{\rm End} \Big(\sharp \Theta_\ep + \sharp(g_\ep \om_X)\Big)  \, \omte^{n} \\
 & = \Theta_\ep \wedge \omte^{n-1}+g_\ep \om_\X \wedge \omte^{n-1}
 \end{align*}
and one obtains that the term ${\rm \bf (III)}$ converges to zero when $\ep,t\to 0$ since
\begin{enumerate}
\item[$\bullet$] $\int_X \Theta_\ep \wedge \omte^{n-1}=c_1(E)\cdotp \{\pi^*\om_X+t\om_\X\}^{n-1}$ and $E$ is exceptional, 
\item[$\bullet$] $\int_X g_\ep \om_\X \wedge \omte^{n-1} \to 0$ when $\ep,t\to 0$ thanks to the smooth convergence to $0$ outside $E$ and the Chern-Levine-Nirenberg inequality combined with the bound \eqref{bounded} on the potentials, cf. first item in Part {\rm \bf (I)}.
\end{enumerate}

\medskip

\textbf{ The term ${\rm \bf (IV)}$.}

\medskip
\noindent
Note that the term $\beta_{t,\ep}\wedge \beta_{t,\ep}^*$ is pointwise negative in the sense of Griffiths on $W$. In particular, the term ${\rm \bf (IV)}$ is non-positive. Since ${\rm \bf (I)}$ and ${\rm \bf (III)}$ converge to zero, this shows that 
\begin{equation}
\label{slope}
\mu(\sF) \le \big(1+\gamma\big( \frac nr-1\big) \big) \cdot \mu(T_\X)
\end{equation}
where the slope is taken with respect to $\pi^*c_1(X)$. 

Working under Assumption~\ref{AAKE}, one obtains the inequality~\eqref{slope} above for any $\gamma>0$ small enough. In particular, this shows that under Assumption~\ref{AAKE}, $T_\X$ is semistable with respect to $\pi^*c_1(X)$.

From now on, we assume that the stronger Assumption~\ref{AKE} holds; i.e. one can choose $\gamma=0$. Assume additionally that there exists a subsheaf $\sF\subset T_\X$ with the same slope as $T_\X$ and let $\sF^{\rm sat}$ be its saturation in $T_\X$; it is a subbundle in codimension one. As the slope has not increased by saturation, $\sF=\sF^{\rm sat}$ in codimension one on $\X\setminus E$. Therefore, if we set $W^\circ:=W\cap (\X\setminus E)$, then $W^\circ \subset \X\setminus E$ has codimension at least two and by the above computation, one has 
$$\lim_{\ep,t\to 0} \int_{W^\circ}  (\beta_{t,\ep}\wedge \beta_{t,\ep}^*\wedge \omte^{n-1})=0.$$
We know by \eqref{limit} that $\beta_{t,\ep} \to \beta_{\infty}$ locally smoothly on $W^\circ$ when $\ep,t \to 0$ where $\beta_{\infty}$ is the second fundamental form induced by the hermitian metric $h_{\rm KE}$ induced by $\pi^*\omke$ on ${T_\X}|_{W^\circ}$ and on $F$ by restriction. By Fatou lemma, we have $\beta_\infty \equiv 0$ on $W^\circ$, that is, we have a holomorphic decomposition $T_\X|_{W^\circ}=F\oplus F^\perp$ where the orthogonal is taken with respect to $h_{\rm KE}$.  \\

We are now ready to prove

\begin{thm}
\label{polystable}
Let $X$ be a $\mathbb Q$-Fano variety.

\begin{enumerate}
\item[$(i)$] If Assumption~\ref{AAKE} is satisfied, then $T_X$ is semistable with respect to $c_1(X)$.

\item[$(ii)$] If Assumption~\ref{AKE} is satisfied, then $T_X$ is polystable with respect to $c_1(X)$. More precisely, we have:

\begin{enumerate}
\item[$\bullet$] Any saturated subsheaf $\sF\subset T_X$ with $\mu(\sF)=\mu(T_X)$ is a direct summand of $T_X$ and $\sF|_{X_{\rm reg}}\subset T_{X_{\rm reg}}$ is a parallel subbundle with respect to $\omke$.
\item[$\bullet$]  There exists a decomposition $$T_X=\bigoplus_{i\in I} \sF_i$$
such that $\sF_i$ is stable with respect to $c_1(X)$ and $\sF_i|_{X_{\rm reg}}\subset T_{X_{\rm reg}}$ is a parallel subbundle with respect to $\omke$.
\end{enumerate}
\end{enumerate}
\end{thm}

\begin{proof}
Let $\sF\subset T_X$ be a subsheaf and let $\alpha:=c_1(X)$. The sheaf $\sF$ induces a subsheaf $\sG^\circ\subset  T_\X|_{\X\setminus E}$ and we denote by $\sG\subset T_\X$ the saturation of $\sG^\circ$ in $T_\X$. By the arguments above (cf. inequality~\eqref{slope} and the comments below it), one has $\mu_{\pi^*\alpha}(\sG) \le \mu_{\pi^*\alpha}(T_\X) = c_1(X)^n/n=\mu_{\alpha}(T_X)$. Moreover, one has clearly $\mu_{\pi^*\alpha}(\sG)=\mu_{\alpha}(\sF)$. This shows that $T_X$ is semistable with respect to $c_1(X)$.  

Now, assume that there exists a Kähler-Einstein metric $\omke$. If $\sF\subset T_X$ satisfies $\mu_{\alpha}(\sF)=0$, then $\mu_{\pi^*\alpha}(\sG)=0$ and we have showed above that $\pi^*\omke$ induces a splitting ${T_\X}|_W=\sG|_W\oplus (\sG|_W)^\perp$ over a Zariski open subset $W\subset \X\setminus E$ whose complement in $\X\setminus E$ has codimension at least two. Set $V:=\pi(W)\subset X_{\rm reg}$ so that  $\sF|_V$ is a subbundle of $T_X$ and we have a splitting 
$T_X|_V= \sF|_V\oplus (\sF|_V)^\perp$ induced by $\omke$ and $\mathrm{codim}_X(X\setminus V) \ge 2$. 

Let us denote by $j:V\hookrightarrow X$ the open immersion. As $\sF\subset T_X$ is saturated, it is reflexive, hence $j_*(\sF|_V) =\sF$. Moreover, $(\sF|_V)^\perp$ extends to a reflexive sheaf $\sF^{\perp}:=j_*((\sF|_V)^\perp)$ on $X$ satisfying $T_X= \sF\oplus \sF^\perp$ on the whole $X$. In particular, $\sF$ is a direct summand of $T_X$ and as such, it is subbundle of $T_X$ over $X_{\rm reg}$. By iterating this process and starting with $\sF$ with minimal rank, one can decompose $T_X=\bigoplus_{i \in I} \sF_i$ into reflexive sheaves which, over $X_{\rm reg}$, are parallel (pairwise orthogonal) subbundles with respect to $\omke$. 
\end{proof}

\section{Semistability of the canonical extension}
\label{semi}
In this section, we keep using the setup and notation of \S~\ref{set-up}. 

\subsection{The canonical extension}\label{section: canonical extension}

Let $\sE$ be a coherent sheaf on $X$ sitting in the exact sequence below
\begin{equation}
\label{ES}
0\longrightarrow \Omega_X^{[1]} \longrightarrow \sE \longrightarrow \sO_X \longrightarrow 0
\end{equation}
The sheaf $\sE$ is automatically torsion-free and it is locally free on $X_{\rm reg}$. 

\begin{rem}
\label{localisation}
Let $U\subset X$ be a non-empty Zariski open subset. As an extension of $\sO_X$ by $ \Omega_X^{[1]}$, $\sE|_U$ is uniquely determined by the image of $1\in H^0(U,\sO_X)$ in $H^1(U, \Omega_X^{[1]})$ under the connecting morphism in the long exact sequence arising from $H^0(U,\--)$. 
\end{rem}

\medskip

From now on, one assumes that the extension class of $\sE$ is the image of $c_1(X)$ in $H^1(X,\Omega_X^1)$ under the canonical map $$\mathrm{Pic}(X)\otimes\mathbb{Q} \simeq H^1(X,\sO_X^*)\otimes\mathbb{Q} \to H^1(X,\Omega_X^1)\to H^1(X,\Omega_X^{[1]}).$$ This is legitimate since $K_X$ is $\mathbb Q$-Cartier. 

\medskip

\begin{dfn}
\label{canonicalextension}
The dual $\sE^*$ of the sheaf $\sE$ sitting in the exact sequence \eqref{ES} with extension class $c_1(X)$ is called the canonical extension of $T_X$ by $\sO_X$. 
\end{dfn}

The exact sequence \eqref{ES} is locally splittable since for any affine $U\subset X$, one has $h^1(U,\Omega_U^{[1]})=0$.
In particular, when one dualizes \eqref{ES}, one see that the canonical extension of $T_X$ by 
$\sO_X$ sits in the short exact sequence below
\begin{equation}
\label{ES2}
0\longrightarrow \sO_X  \longrightarrow \sE^* \longrightarrow T_X \longrightarrow 0.
\end{equation}

\medskip 

The goal of this section is to prove the following, cf. Theorem~\ref{thm2}.

\begin{thm}
\label{semistable}
Let $X$ be a $\Q$-Fano variety. If Assumption~\ref{AAKE} is satisfied, then the canonical extension $\sE^*$ of $T_X$ by $\sO_X$ is semistable with respect to $c_1(X)$.
\end{thm}

The proof of Theorem~\ref{semistable} above is divided into two main steps corresponding to the next two sub-sections. First one can reduce the statement above to a stability property on the resolution $\X$ thanks to Lemma~\ref{reduction2} and then we prove the said statement, cf. Theorem~\ref{gang}.  

\subsection{Reduction to a statement on the resolution}
\label{reduction}
\noindent 
Let $\wh \sE$ be the vector bundle on $\X$ sitting in the exact sequence below
\begin{equation}
\label{ES3}
0\longrightarrow \Omega_\X^{1} \longrightarrow \wh \sE \longrightarrow \sO_\X \longrightarrow 0
\end{equation}
such that its extension class is $\pi^*c_1(X) \in H^1(\X,\Omega_\X^1).$
Its dual sits in the exact sequence 
\begin{equation}
\label{ES4}
0\longrightarrow\sO_\X \longrightarrow \wh \sE^* \longrightarrow T_\X \longrightarrow 0.
\end{equation}

\begin{lem}
\label{reduction2}
If the vector bundle $\wh\sE^*$ is semistable with respect to $\pi^*c_1(X)$, then the torsion-free sheaf $\sE^*$ is semistable with respect to $c_1(X)$. 
\end{lem}

\begin{proof}
Set $\alpha:=c_1(X)$.
Let $X^\circ\subseteq X_{\rm reg}$ be an open set with complement of codimension at least $2$ in $X$ such that the restriction $\pi_{|\X^\circ}$ of $\pi$ to $\X^\circ :=\pi^{-1}(X^\circ)$ induces an isomorphism $\X^\circ \simeq X^\circ$. By Remark~\ref{localisation} 
we have 
\begin{equation}
\label{identification}
(\pi^*\sE^*)_{|\X^\circ} \simeq \mathcal {\wh\sE^*}_{|\X^\circ}.
\end{equation}
Let $\sF\subseteq \sE^*$ be a subsheaf and let $\wh\sF \subseteq \wh\sE^*$ be the saturated subsheaf of $\wh\sE^*$ whose restriction to 
$\X^\circ$ is $(\pi^*\sF)_{|\X^\circ}$. By the projection formula together with the fact that $X \setminus X^\circ$ has codimension at least $2$ in $X$, we have
$$\mu_\alpha(\sF)=\mu_{\pi^*\alpha}(\wh\sF) \quad \textup{and} \quad \mu_\alpha(\sE^*)=\mu_{\pi^*\alpha}(\wh\sE^*).$$
The lemma follows easily.
\end{proof}

\subsection{Statement on the resolution}
In this section, we prove that the vector bundle $\wh \sE^*$ from \S~\!\ref{reduction} is semistable with respect to $\pi^*c_1(X)$, cf. Theorem~\ref{gang} below. In order to streamline the notation, we set $\sV:=\wh\sE^*$ and in the following we will not distinguish between the locally free sheaf $\sV$ and the associated vector bundle. Recall that $\sV$  fits into the exact sequence of locally free sheaves 
\begin{equation}\label{bond1}
0\longrightarrow \sO_\X\longrightarrow \sV\to T_\X\longrightarrow 0,
\end{equation}
We denote by $\beta\in H^1(\X, T_\X^\star)$ the second fundamental form.

\medskip

\noindent Our result in this section is a singular version of Theorem 0.1 in \cite{tian92}.

\begin{thm}
\label{gang} 
Let $X$ be a $\Q$-Fano variety satisfying Assumption~\ref{AAKE}. Let $\sV$ be the vector bundle on $\X$ appearing in \eqref{bond1}, whose extension class $\beta$ coincides with the inverse image of the 
  first Chern class of $X$ by the resolution $\pi:\X\to X$.
  
  \noindent
  Then $\sV$ is semistable with respect to $\pi^\star c_1(X)$.
\end{thm}
\medskip

\begin{proof} 
 The strategy of proof is as follows.  We would like to compute the slope of $\sF$ using an hermitian metric on $\sV$ induced by the (twisted) Kähler-Einstein metric, using an approximation process as in \S~\ref{stab}. As the natural metric in the extension class of $\sV$ is singular, we introduce an algebraic $1$-parameter family $(\sV_z)_{z\in \mathbb C}$ that can be endowed with natural smooth hermitian metrics for suitable $z \in \mathbb R$ close to zero and such that we have sheaf injections $\sV\hookrightarrow \sV_t \otimes \sO_\X(E)$. We then proceed to compute slopes following the strategy of \S~\ref{stab}.

\medskip

\textbf{Step 1. Deformations of $\sV$.}

\medskip
We pick an arbitrary subsheaf $\sF\subseteq \sV$ of the vector bundle $\sV$ sitting in the exact sequence below
$$0 \to \sO_\X \to \sV \to T_\X \to 0$$
and corresponding to the extension class
$$\alpha=(a_{ij})\in \mathrm{Ext}^1(T_\X,\sO_\X) \simeq
H^1(\X,{\sH}om(T_\X,\sO_\X))$$ relatively to a covering by open subsets $(U_i)$. The bundle $\sV$ can be obtained as follows: on $U_i$, it is the trivial extension, $\sV_{|U_i}={\sO_\X}_{|U_i}\oplus {T_\X}_{|U_i}$ and the transition functions are given by  
$$
\begin{pmatrix} \mathrm{Id}_{\sO_\X}|_{U_{ij}} & a_{ij} \\ 0 & \mathrm{Id}_{T_\X}|_{U_{ij}} \end{pmatrix}.
$$
The subsheaf $\sF$ is given by two morphisms of sheaves $p_i\colon \sF_{|U_i} \to {\sO_\X}_{|U_i}$ and $q_i\colon \sF_{|U_i} \to {T_\X}_{|U_i}$ satisfying 
$$\begin{cases}
p_i|_{U_{ij}}=p_j|_{U_{ij}}+a_{ij}\circ (q_j|_{U_{ij}}),\\
q_i|_{U_{ij}}=q_j|_{U_{ij}}.
\end{cases}$$

Recall that we have a reduced divisor $E=E_1+\cdots+E_r$. Up to refining the covering $(U_i)$, one can assume that 
$E_k$ is given by the equation $f_{ki}=0$ on $U_i$. The transition functions of $\sO_\X(E_k)$ are $g_{k,ij}=\frac{f_{kj}}{f_{ki}}$.

Now, given complex numbers $z_1, \ldots, z_r \in \mathbb C$, one considers the extension $\sV_{z_1,\ldots,z_r}$ of $T_\X$ by $\sO_\X$ whose class is 
$$\alpha+z_1[\frac{d g_{1,ij}}{g_{1,ij}}]+\cdots + z_r[\frac{dg_{r,ij}}{g_{r,ij}}]=\alpha+\sum_k z_kc_1(E_k).$$ 

Set $\sV_{z_1, \ldots, z_r}(E):=\sV_{z_1, \ldots, z_r} \otimes \sO_\X(E)$. Then, there is an injection of sheaves 
$$\sF \subseteq \sV_{z_1,\ldots,z_s}(E)$$ 
extending $\sF \subseteq \sV \subseteq \sV(E)$ for $(z_k)$ in a Zariski open neighborhood of $0\in \mathbb C^r$.

Indeed, consider the morphism $\sF_{|U_i} \to \sV_{z_1,\ldots,z_s}(E)|_{U_i}$ given by $p_i+\sum_k z_k \frac{df_{ki}}{f_{ki}} \circ q_i$ 
on the first factor and $q_i$ on the second. Those morphisms can be glued since one has

$$\frac{df_{ki}}{f_{ki}}=\frac{dg_{k,ij}}{g_{k,ij}} + \frac{df_{kj}}{f_{kj}},$$
for any index $k$. The induced map $\sF \to \sV_{z_1,\ldots,z_s}(E)$ is obviously injective for $(z_k)$ in a Zariski open neighborhood of $0\in \mathbb C^r$.\\

Now, recall that $\alpha=\pi^*c_1(X)$ and that the Kähler metric $\om_\X$ lives in the class $\alpha -\sum \ep_k c_1(E_k)$ for some $\ep_k>0$, so that the approximate Kähler-Einstein metric $\omte \in (1+t)\alpha_t$ where  $$\alpha_t:=\alpha-\frac{t}{1+t}\sum_k \ep_k c_1(E_k).$$
For any $t\in \mathbb  R$, we set 
$$\sV_t:=V_{z_1, \ldots,z_r} \quad \mbox{and} \quad \sV_t(E):=\sV_t \otimes \sO_\X(E)$$
where $z_k:=-\frac {t}{1+t}\cdot \ep_k$ for $1\le k \le r$. This vector bundle $\sV_t$ is the extension of $T_\X$ by $\sO_\X$ with extension class $\alpha_t$ and $\sV_t(E)$ comes equipped with a sheaf injection
\begin{equation}
\label{subsheaf}
\sF \subseteq {\sV}_t(E).
\end{equation}
Moreover, it is clear from the definition of $\sV_{z_1, \ldots, z_r}$ that we have 
\begin{equation}
\label{c1}
c_1(\sV_t(E))=c_1(\sV)+c_1(E)
\end{equation}
for any $t\in \R$. \\

\textbf{Step 2. Metric properties of $\sV_t(E)$. }

\medskip
First of all, we pick one number $\gamma>0$ as in Assumption~\ref{AAKE}. It will be fixed until the very end of the argument.

We seek to endow $\sV_t(E)$ with a suitable smooth hermitian metric, at least when $t>0$ is small enough. Given that $\sV_t(E)=\sV_t \otimes \sO_\X(E)$ and that we have already fixed a smooth hermitian metric $h_E$ on $\sO_\X(E)$ in \eqref{he}, it is enough to construct a hermitian metric on $\sV_t$.

Now, we can endow the bundles $\sO_{\X}$ and $T_\X$ with the trivial metric and the hermitian metric $h_{t,\ep}$ induced by $\omte$, respectively. Now, we set 
$$\beta_t = \frac 1{1+t} \omte \in \alpha_t$$
which we view as an element of $\cC^{\infty}_{0,1}(\X,T_\X^*)$. Relatively to a fixed $\mathcal C^{\infty}$ splitting of $\sV_t$, the direct sum metric $h_{\sV_t}$ induced on $\sV_t$ has a Chern connection $D_{\sV_t}$ which has the following expression
$$D_{\sV_{t}}=\begin{pmatrix}
d & -\beta_t \\
\beta_t^*& D_{T_\X}
\end{pmatrix}
$$
or equivalently
\begin{equation}\label{bond19}
D_{\sV_{t}}(s_1, s_2)= \left(d s_1-\beta_t\cdot s_2, \beta^\star_t\cdot s_1+ D_{T_\X}s_2\right)
\end{equation}
where $D_{T_\X}$ is the Chern connections induced by $h_{t,\ep}$ on $T_\X$. Of course, it depends strongly on the parameters $t,\ep$. We denote by $\beta^\star_t \in \cC^{\infty}_{1,0}(\X,T_\X)$ the adjoint of $\beta_t\in \cC^{\infty}_{0,1}(\X,T_\X^*)$. Moreover, the Chern curvature of $D_{\sV_t}$ is given by
$$\Theta(\sV_t,h_{\sV_t})=\begin{pmatrix}
 - \beta_t \wedge\beta_t^* & D'_{T_\X^*} \beta_t \\
\overline \partial \beta_t^*& \Theta(T_\X, h_{t,\ep})- \beta_t^*\wedge \beta_t
\end{pmatrix}
$$
where $D'_{T_\X^*}$ is the $(1,0)$-part of the Chern connection of $(T^\star_\X,h_{t,\ep}^*)$.

\smallskip

\noindent We analyze next  several quantities which are playing a role in the evaluation of the curvature of $\sV_{t}$.\\

$\bullet $ \textit{The factor $\beta_t$}.

\medskip

\noindent
The form $\beta_t$ is given by
\begin{equation}\label{bond11}
  \beta_t= \frac{1}{1+t}\sum \omega_{p\overline q}\left(\frac{\partial}{\partial z_p}\right)^\star\otimes
  dz_{\overline q}
\end{equation}
where $\omega_{p\overline q}$ are the coefficients of $\omte$ with respect to the
coordinates $(z_i)_{i=1,\dots,n}$. Its adjoint is computed by the formula
\begin{equation}\label{bond20}
\langle\beta_t \cdot v, w \rangle+ \langle v, \beta^\star_t \cdot w \rangle= 0,
\end{equation}
where the first bracket is the standard hermitian product in $\mathbb C$ and the second one is the one induced by $(T_\X , h_{t,\ep})$. We have
\begin{equation}\label{bond12}
\beta^\star_t= -\frac{1}{1+t}\sum \frac{\partial}{\partial z_i}\otimes
  dz_{i}.
\end{equation} 
We have the following formulas
\begin{equation}\label{bond13}
D^\prime_{T_\X^*} \beta_t= 0, \qquad \overline\partial \beta^\star_t= 0.
\end{equation}
The first equality holds since $\omte$ is a K\"ahler metric while the second one is obvious from \eqref{bond12}.

\noindent Moreover, we have

\begin{equation}\label{bond14}
(1+t)^2 \cdot \beta_t\wedge \beta_t^\star\wedge \omte^{n-1}= -\frac{1}{n}\omte^n
\end{equation}  
as well as
\begin{equation}\label{bond15}
(1+t)^2 \cdot \beta_t^\star\wedge \beta_t= \omte \otimes \mathrm{Id}_{T_\X}.
\end{equation}
\medskip

$\bullet $ \textit{The curvature of $\sV_{t}$}.

\medskip

If we replace $\beta_t$ by $(1+t) \sqrt\mu \beta_t$ for some positive number $\mu$, this does not affect the complex structure of the bundles at stakes but only the metrics. Moreover, we see from the identities \eqref{bond13}-\eqref{bond14}-\eqref{bond15} that the curvature becomes
\begin{equation*}
\Theta(\sV_{t},h_{\sV_{t}}) \wedge \omte^{n-1}=
\begin{pmatrix}
\frac \mu n \omte^n&0 \\
0 & \Theta(T_\X,h_{t,\ep}) \wedge \omte^{n-1}- \mu \omte^n \otimes \mathrm{Id}_{T_\X}
\end{pmatrix}\\
\end{equation*}
Now we choose $\mu$ so that $\frac \mu n=1-\mu$, i.e. $\mu:=\frac n {n+1}$. Recalling \eqref{contraction} and the expression of the Ricci curvature of $\omte$ given in \eqref{ricci}, we get that
\begin{equation*}
\Theta(T_\X,h_{t,\ep}) \wedge \omte^{n-1}- \mu \omte^n \otimes \mathrm{Id}_{T_\X} = \frac 1{n+1} \omte^n \otimes \mathrm{Id}_{T_\X} +A_{t,\ep,\gamma} \omte^n
\end{equation*}
where 
\begin{equation}
\label{ateg}
A_{t,\ep,\gamma}=-\gamma \mathrm{Id}_{T_\X} +\sharp\big[\gamma\pi^*\om_X-t\om_\X+(1-\gamma)dd^c(\psi_\ep-\vpte)-\Theta_\ep\big]
\end{equation}
is such that the number $$a_{t,\ep,\gamma}:=\frac 1 n\int_\X \mathrm{tr}_{\rm End} \mathrm{pr}_F (A_{t,\ep, \gamma})|_F \, \omte^n$$ satisfies
 \begin{equation}
 \label{aep}
 \limsup_{\gamma \to 0} \limsup_{t\to 0}\limsup_{\ep \to 0} a_{t,\ep, \gamma} = 0
 \end{equation}
thanks to the computations of \S~\!\ref{stab}. \\

\medskip

$\bullet $ \textit{The curvature of $\sV_{t}(E)$}.

\medskip

Finally, we endow $\sV_t(E)$ with the metric $h_{\sV_t(E)}:=h_{\sV_t} \otimes h_E$. It satisfies 
\begin{equation}
\label{curvatureformula}
\Theta(\sV_{t}(E),h_{\sV_{t}(E)}) \wedge \omte^{n-1}= 
\frac 1{n+1} \omte^n \otimes \mathrm{Id}_{\sV_{t}}
+A_{t,\ep,\gamma} \omte^n+(\Theta_E\wedge \omte^{n-1}) \otimes \mathrm{Id}_{\sV_{t}(E)}.
\end{equation}
where 
$A_{t,\ep,\gamma}$ is defined in \eqref{ateg} and satisfies \eqref{aep}.\\

\textbf{Step 3. The slope inequality. }

\medskip
Now, one wants to follow the strategy in \S~\ref{stab} and compute the slope of $\sF$ using the induced metric $h_{F_t}$ from $(\sV_{t}(E),h_{\sV_{t}(E)})$ under the sheaf injection \eqref{subsheaf}. The metric $h_{F_t}$ is well-defined only on the locus $W\subset \X$ where $F_t:=\sF|_{W}$ is a subbundle. As $\sF$ may not be saturated in $\sV_t(E)$, the complement of $W$ may have codimension one. However, we have the formula
\begin{align*}
\mu_{\omte}(\sF) &= \frac 1r \int_{W} c_1(F_t,h_{F_t}) \wedge \omte^{n-1} -c_1(D) \cdot \{\omte\}^{n-1}\\
& \le \frac 1r \int_{W} c_1(F_t,h_{F_t}) \wedge \omte^{n-1} \\
&\le   \mu_{\omte}(\sV_{t}(E)) +a_{t,\ep, \gamma}+ c_1(E)\cdot \{\omte\}^{n-1}
\end{align*}
where $D$ is an effective divisor such that $\sO_X(D)=\det((\sV_{t}(E)/\sF)_{\rm tor})$.
Since $E$ is $\pi$-exceptional, the conclusion follows from the curvature formula \eqref{curvatureformula} along with \eqref{aep} and the two easy facts below
\begin{enumerate}
\item[$\bullet$] $\mu_{\omte}(\sF) \to \mu_{\alpha}(\sF)$ when $t\to 0$,
\item[$\bullet$] $\mu_{\omega_{t,\ep}}(\sV_{t}(E)) \to \mu_{\alpha}(\sV)$ when $t,\ep\to0$ since $E$ is exceptional, cf. \eqref{c1}.
\end{enumerate}
Theorem~\ref{gang} is now proved. 
\end{proof}

\begin{rem}
\label{rem:polystable}
If $\mu_\alpha(\sF)=0$ and assuming Assumption~\ref{AKE}, then the same arguments as in the end of Section~\ref{stab} show the the orthogonal complement of $\sF \subset \sV_0(E)$ with respect to $h_{\sV_0(E)}$ on $\X\setminus E$ is holomorphic. Thanks to \eqref{identification} in Lemma~\ref{reduction2}, this shows that the canonical extension $\sE^*$ is polystable with respect to $c_1(X)$. 
\end{rem}

\section{A splitting theorem}

\subsection{Foliations} In this section, we recollect some results about foliations that we will use later on for the reader's convenience.
We refer to \cite[Sec. 3 and 4]{cd1fzerocan} and the references therein for notions around foliations on normal varieties and their singularities.

Here we only recall the notion of weakly regular foliation.
Let $\sF$ be a foliation of positive rank $r$ on a normal variety $X$. The $r$-th wedge product of the inclusion $\sF \subseteq T_X$ gives a map $$\sO_X(-K_\sF) \inj (\wedge^rT_X)^{**}.$$ We will refer to the dual map 
$$\Omega_X^{[r]} \to \sO_X(K_\sF)$$ as the \textit{Pfaff field} associated to $\sF$.
The foliation $\sF$ is called \textit{weakly regular} if the induced map
$$ (\Omega_X^r\otimes\sO_X(-K_\sF))^{**}\to \sO_X$$
is surjective (see \cite[Sec. 5.1]{cd1fzerocan}).

\medskip

Examples of weakly regular foliations are provided by the following result (see \cite[Lem. 5.8]{cd1fzerocan}).

\begin{lem}\label{lemma:direct_summand_regular}
Let $X$ be a normal variety, and let $\sF$ be a foliation on $X$. Suppose that there exists a distribution $\sG$ on $X$ such that 
$T_X=\sF\oplus \sG$. Then $\sF$ is weakly regular.
\end{lem}

The following lemma says that a weakly regular foliation has mild singularities if its canonical divisor is Cartier and the ambient space has klt singularities (see \cite[Lem. 5.9]{cd1fzerocan}).

\begin{lem}\label{lemma:regular_versus_canonical}
Let $X$ be a normal variety with klt singularities, and let $\sF$ be a foliation on $X$. 
Suppose that $K_\sF$ is Cartier.
If $\sF$ is weakly regular, then it has canonical singularities.
\end{lem}

Next, we recall the behaviour of weakly regular foliations with respect to finite covers (see \cite[Prop. 5.13]{cd1fzerocan}).

\begin{lem}\label{lemma:regular_quasi_etale}
Let $X$ be a normal variety, let $\sF$ be a foliation on $X$, and let $f\colon X_1 \to X$ be a finite cover.
Suppose that each codimension $1$ irreducible component of the branch locus of $f$ is $\sF$-invariant. Then 
$\sF$ is weakly regular if and only if $f^{-1}\sF$ is weakly regular.
\end{lem}

Finally, we recall the behaviour of foliations with canonical singularities with respect to finite covers and birational maps (see \cite[Lem. 4.3]{cd1fzerocan}).

\begin{lem}\label{lemma:canonical_quasi_etale_cover}
Let $f\colon X_1 \to X$ be a finite cover of normal varieties, and let $\sF$ be a foliation on $X$ with
$K_\sF$ $\mathbb{Q}$-Cartier. Suppose that each codimension $1$ component of the branch locus of $f$ is 
$\sF$-invariant. If $\sF$ has canonical singularities, then $f^{-1}\sF$ has canonical singularities as well.
\end{lem}

\begin{lem}\label{lemma:singularities_birational_morphism}
Let $q\colon Z \to X$ be a birational quasi-projective morphism of normal varieties, and let $\sF$ be a foliation on $X$.
Suppose that $K_\sF$ is $\mathbb{Q}$-Cartier and that $K_{q^{-1}\sF}\sim_\mathbb{Q} q^*K_\sF$.
If $\sF$ has canonical singularities, then $q^{-1}\sF$ has canonical singularities as well.
\end{lem}

\begin{proof}
By assumption, there exist a normal variety $\wb Z \supseteq Z$ and a projective birational morphism $\wb q \colon \wb Z \to X$ whose restriction to $Z$ is $q$. The same argument used in the proof of \cite[Lem. 4.2]{cd1fzerocan} shows that 
$$a(E,\wb Z,{\wb q}^{-1}\sF)=a(E,X,\sF)$$
for any exceptional prime divisor $E$ over $\wb{Z}$ with non-empty center in $Z$. The lemma follows easily.
\end{proof}

\subsection{Weakly regular foliations with algebraic leaves} This section contains a generalization of Theorem 6.1 in \cite{cd1fzerocan}. 
The following result is proved in \textit{loc. cit.} under the additional assumption that $\sF$ has canonical singularities.

\begin{thm}\label{thm:regular_foliation_morphism}
Let $X$ be a normal projective variety with $\mathbb{Q}$-factorial klt singularities, and let $\sF$ be a weakly regular foliation on $X$ with algebraic leaves. 
\begin{enumerate}
\item\label{thm:regular_foliation_morphism1}Then $\sF$ is induced by a surjective equidimensional morphism $p\colon X \to Y$ onto a normal projective variety $Y$.
\item\label{thm:regular_foliation_morphism2}Moreover, there exists an open subset $Y^\circ$ with complement of codimension at least $2$ in $Y$ such that $p^{-1}(y)$ is irreducible for any $y\in Y^\circ$.
\end{enumerate}
\end{thm}

Before we give the proof of Theorem \ref{thm:regular_foliation_morphism}, we need to prove a number of auxiliary statements. 

Throughout the present section, we will be working in the following setup.

\begin{setup}
\label{setup:fol}
Let $X$ and $Y$ be normal quasi-projective varieties, and let $p'\colon X \map Y$ be a dominant rational map with $r:=\dim X - \dim Y >0$. 
Let $Z$ be the normalization of the graph of $p'$, and let $p\colon Z \to Y$ and $q\colon Z \to X$ be the natural morphisms.
Let $\sF$ be the foliation induced by $p'$. 
\end{setup}

\begin{prp}\label{prop:family_leaves_canonical_bundle_formula}
Setting and notation as in \ref{setup:fol}, and assume that $K_\sF$ is Cartier.
\begin{enumerate}
\item\label{prop:family_leaves_canonical_bundle_formula1} Then the Pfaff field $\Omega_X^{[r]} \to \sO_X(K_\sF)$ associated to $\sF$ induces a map $$\Omega_Z^{[r]} \to q^*\sO_X(K_\sF)$$ which factors through the Pfaff field $\Omega_Z^{[r]} \to \sO_Z(K_{q^{-1}\sF})$ associated to $q^{-1}\sF$. In particular, there exists an 
effective $q$-exceptional Weil divisor $B$ on $Z$ such that $$K_{q^{-1}\sF}+B\sim_{\mathbb{Z}}q^*K_\sF.$$
\item\label{prop:family_leaves_canonical_bundle_formula2} Moreover, if $E$ is a $q$-exceptional prime divisor on $Z$ such that $p(E)=Y$, then $E \subseteq \textup{Supp}\,B$.
\end{enumerate}
\end{prp}

\begin{proof}
Let $Z_0\subseteq Y \times X$ be the graph of $p'$, and denote by $n\colon Z \to Z_0$ the normalization map. Consider the foliation 
$$\sG:=\textup{pr}_X^*\sF \subseteq \textup{pr}_X^*T_X \subseteq \textup{pr}_Y^*T_Y\oplus \textup{pr}_X^*T_X.$$Let $\Omega_X^{r}\to\sO_X(K_\sF)$ be the map induced by the Pfaff field 
$\Omega_X^{[r]}\to\sO_X(K_\sF)$. 
By construction, $Z_0$ is invariant under $\sG$, and hence, there is a factorization:

\begin{center}
\begin{tikzcd}
{\Omega_{Y\times X}^{r}}|_{Z_0} \ar[r]\ar[d, two heads] & \textup{pr}_X^*\Omega^r_X|_{Z_0} \ar[r] & (\textup{pr}_X^*{\sO_X(K_\sF)})|_{Z_0} \ar[d, equal]\\
\Omega_{Z_0}^{r} \ar[rr]  && \sO_{Y\times X}(K_\sG)|_{Z_0}.
\end{tikzcd}
\end{center}
Notice that the foliation induced by $\sG$ on $Z$ is $q^{-1}\sF$.
By \cite[Prop. 4.5]{adk08}, the map $\Omega_{Z_0}^{r} \to (\textup{pr}_X^*{\sO_X(K_\sF)})|_{Z_0}$ extends to a map
$$\Omega_{Z}^{r} \to n^*((\textup{pr}_X^*{\sO_X(K_\sF)})|_{Z_0})\simeq q^*\sO_X(K_\sF),$$ which gives a morphism
$$\Omega_{Z}^{[r]} \to q^*\sO_X(K_\sF).$$
This map factors through the Pfaff field $$\nu_Z\colon\Omega_{Z}^{[r]} \to \sO_Z(K_{q^{-1}\sF})$$ associated to $q^{-1}\sF$
away from the closed set where $\nu_Z$ is not surjective, which has codimension at least $2$ in $Z$. Hence, there exists an 
effective Weil divisor $B$ on $Z$ such that $$K_{q^{-1}\sF}+B\sim_{\mathbb{Z}}q^*K_\sF.$$ 
Moreover, the morphism $\Omega_{Z}^{[r]} \to q^*\sO_X(K_\sF)$ identifies with the composed map
$$\Omega_{Z}^{[r]} \to \sO_Z(K_{q^{-1}\sF})\to q^*\sO_X(K_\sF)$$
since $q^*\sO_X(K_\sF)$ is torsion-free.
Note that $B$ is obviously $q$-exceptional, proving the first item.

\medskip

The second item follows from \cite[Lem. 4.19]{cd1fzerocan} by induction on the rank of $\sF$ as in the proof of Proposition 4.17 in \textit{loc. cit.} Notice that the assumption that the birational morphism is projective in the statement of Lemma 4.19 in \textit{loc. cit.} is not necessary.
\end{proof}

\begin{cor}\label{cor:weakly_regular_universal_family_leaves}
Setting and notation as in \ref{setup:fol}. Suppose that $X$ has klt singularities.
Suppose in addition that $K_\sF$ is Cartier and that $\sF$ is weakly regular.
\begin{enumerate}
\item\label{cor:weakly_regular_universal_family_leaves1} Then the foliation $q^{-1}\sF$ is weakly regular and $K_{q^{-1}\sF}\sim_\mathbb{Z}q^*K_\sF$.
\item\label{cor:weakly_regular_universal_family_leaves2} Moreover, if $E$ is a prime $q$-exceptional divisor on $Z$, then $p(E) \subsetneq Y$.
\end{enumerate}
\end{cor}

\begin{proof}
By item~\ref{prop:family_leaves_canonical_bundle_formula1} in Proposition~\ref{prop:family_leaves_canonical_bundle_formula}, 
the Pfaff field $$\Omega_X^{[r]}\to\sO_X(K_\sF)$$ associated to $\sF$ induces a map
$$\Omega_{Z}^{[r]} \to q^*\sO_X(K_\sF)$$ which factors through the Pfaff field $\Omega_Z^{[r]} \to \sO_Z(K_{q^{-1}\sF})$ associated to $q^{-1}\sF$. On the other hand, by \cite[Thm. 1.3]{kebekus_pull_back}, there exists a morphism of sheaves
$$q^*\Omega_X^{[r]} \to \Omega_Z^{[r]}$$ that agrees with the usual pull-back morphism of K\"ahler differentials wherever this makes sense. One then readily checks that we obtain a commutative diagram as follows:

\begin{center}
\begin{tikzcd}
q^*\Omega_X^{[r]} \ar[r, two heads]\ar[d] &  q^*\sO_X(K_\sF) \ar[d, equal]\\
\Omega_{Z}^{[r]} \ar[r]  & q^*\sO_X(K_\sF).
\end{tikzcd}
\end{center}
This implies that the map $\Omega_{Z}^{[r]} \to q^*\sO_X(K_\sF)$ is surjective. Consequently, this map identifies with Pfaff field
associated to $q^{-1}\sF$, proving item 1.

Finally, item 2 is an immediate consequence of item 1 together with item~\ref{prop:family_leaves_canonical_bundle_formula2} in Proposition~\ref{prop:family_leaves_canonical_bundle_formula}.
\end{proof}

As we will see, Theorem \ref{thm:regular_foliation_morphism} is an easy consequence of Lemma \ref{lemma:irreducible_fibers} and 
Lemma \ref{lemma:small} below.

\begin{lem}\label{lemma:irreducible_fibers}
Setting and notation as in \ref{setup:fol}. Suppose that $X$ has klt singularities and that $\sF$ is weakly regular.
Then there exists an open subset $Y^\circ$ with complement of codimension at least $2$ in $Y$ such that, for any $y\in Y^\circ$,
either $p^{-1}(y)$ is empty or any connected component of $p^{-1}(y)$ is irreducible.
\end{lem}

\begin{proof} We argue by contradiction and assume that there exists a prime divisor $D \subset Y$ such that, for a general point $y\in D$,
$p^{-1}(y)$ is non-empty and some connected component of $p^{-1}(y)$ is reducible. Let $S \subseteq p^{-1}(D)$ be a subvariety of codimension $2$ in $Z$ such that for a general point $z \in S$ there is at least two irreducible components of $p^{-1}(p(z))$ passing through 
$z$.

\medskip

\textbf{Step 1. Construction.} 

\medskip

Shrinking $Y$ if necessary, we may assume without loss of generality that $p$ is equidimensional.
Replacing $X$ by an open neighborhood of the generic point of $q(S)$, we may also assume that there exists a positive integer $m$ such that 
$$\sO_X(m{K_\sF})\simeq \sO_X.$$
Let $f \colon X_1 \to X$ be the associated cyclic cover, which is quasi-\'etale (see \cite[Def. 2.52]{KM}), and let $Z_1$ be the normalization of the product $Z \times_{X} X_1$. The induced morphism 
$g\colon Z_1 \to Z$ is then a finite cover.

By \cite[Lem. 4.2]{druel15}, there exists a finite cover
$Y_2 \to Y$ with $Y_2$ normal and connected such that the following holds. If 
$Z_2$ denotes the normalization of the product $Y_2 \times_Y Z_1$, then the natural morphism $p_2\colon Z_2 \to Y_2$ has reduced fibers over codimension $1$ points in $Y_2$. We may also assume that $Y_2 \to Y$ is a Galois cover.
We obtain a commutative diagram as follows:

\begin{center}
\begin{tikzcd}[row sep=large, column sep=large]
Z_2  \ar[dd, "{p_2}"']\ar[r, "{g_1}"]  & Z_1 \ar[dd, bend right, "{p_1}"']\ar[d, "{g}"]\ar[r, "{q_1}"]  & X_1\ar[d, "f"]\\
  & Z \ar[d, "{p}"]\ar[r, "{q}"]  & X\\
Y_2 \ar[r]& Y. &
\end{tikzcd}
\end{center}
Notice that $g \circ g_1\colon Z_2 \to Z$ is a finite Galois cover.

\medskip

\textbf{Step 2. Away from a closed subset of codimension at least $3$, $Z$ has quotient singularities and the foliation induced by $p$ on $Z$ is weakly regular.}

\medskip

Notice that $X_1$ has klt singularities by \cite[Prop. 3.16]{Kollar97}, and that the foliation $\sF_{X_1}:=f^{-1}\sF$ is weakly regular by Lemma \ref{lemma:regular_quasi_etale}. 
Observe now that the foliation 
$\sF_{Z_1}:=q_1^{-1}\sF_{X_1}$ is given by $p_1$ and that $Z_1$ identifies with the normalization of the graph of the rational map $p_1\circ q_1^{-1}$. Therefore, 
$\sF_{Z_1}$ is weakly regular and 
$$K_{\sF_{Z_1}}\sim_\mathbb{Z}q_1^*K_{\sF_{X_1}}$$
by item~\ref{cor:weakly_regular_universal_family_leaves1} in Corollary~\ref{cor:weakly_regular_universal_family_leaves}. 
On the other hand, $\sF_{X_1}$ has canonical singularities (see Lemma \ref{lemma:regular_versus_canonical}). 
Applying Lemma \ref{lemma:singularities_birational_morphism}, we conclude that 
$\sF_{Z_1}$ has canonical singularities as well. This in turn implies that the foliation $\sF_{Z_2}:=g_2^{-1}\sF_{Z_1}$ has also canonical singularities (see Lemma \ref{lemma:canonical_quasi_etale_cover}).
From \cite[Lem. 5.4]{druel15}, we conclude that $Z_2$ has canonical singularities over a big open set contained in $Y_2$, using the fact that $p_2$ has reduced fibers over codimension $1$ points by construction. In particular, $Z_2$ has canonical singularities in codimension $2$.

Since $g \circ g_1\colon Z_2 \to Z$ is a finite Galois cover, there exists an effective $\mathbb{Q}$-divisor 
$\Delta$ on $Z$ such that $$K_{Z_2}\sim_\mathbb{Q}(g \circ g_1)^*(K_Z+\Delta).$$ Moreover, away from a closed subset of codimension at least $3$,
$K_Z+\Delta$ is $\mathbb{Q}$-Cartier by \cite[Lemma 2.6]{cd1fzerocan}), and the pair $(Z,\Delta)$ is klt by \cite[Prop. 3.16]{Kollar97}.

\medskip

By construction, any irreducible codimension $1$ component of the ramification locus of $g$ is $q_1$-exceptional, and hence
invariant under $\sF_{Z_1}$ by item~\ref{cor:weakly_regular_universal_family_leaves2} in Corollary~\ref{cor:weakly_regular_universal_family_leaves}.
It follows from Lemma \ref{lemma:regular_quasi_etale} that $\sF_Z:=q^{-1}\sF$ is weakly regular in codimension $2$.

\medskip

\textbf{Step 3. End of proof.}

\medskip

Let $z \in S$ be a general point.
Recall from \cite[Prop. 9.3]{GKKP} that $z$
has an analytic neighborhood $U\subseteq Z$ that is biholomorphic to an analytic neighborhood of the origin in a variety of the form 
$\mathbb{C}^{\dim Z}/G$, where G is a finite subgroup of $\textup{GL}(\dim Z,\mathbb{C})$
that does not contain any quasi-reflections. In particular, if $W$ denotes the inverse image of $U$ in the affine space 
$\mathbb{C}^{\dim Z}$,
then the quotient map $$g_U\colon W\to W/G\simeq U$$ is \'etale outside of the singular set. 

By Lemma \ref{lemma:regular_quasi_etale} again, $\sF_{Z}$ induces a regular foliation on $W$.
Let $F_1$ and $F_2$ be irreducible components of 
$p^{-1}(p(z))$
passing through $z$ with $F_1\neq F_2$. 
Note that $$g_U^{-1}(F_1\cap U)\cap g_U^{-1}(F_2\cap U)\neq\emptyset.$$
By general choice of $z$, $F_1$ and $F_2$ are not contained in the singular locus of $\sF_{Z}$,
and hence both $g_U^{-1}(F_1\cap U)$ and $g_U^{-1}(F_2\cap U)$ are a disjoint union of leaves.
But then, any leaf passing through some point of $g_U^{-1}(F_1\cap U)\cap g_U^{-1}(F_2\cap U)$
is a connected component of both $g_U^{-1}(F_1 \cap U)$ and $g_U^{-1}(F_2\cap U)$.
This in turn implies that $F_1=F_2$, yielding a contradiction. This finishes the proof of the lemma.
\end{proof}

\begin{lem}\label{lemma:small}
Setting and notation as in \ref{setup:fol}. Suppose that $X$ has klt singularities and that $\sF$ is weakly regular.
Let $E$ be a prime $q$-exceptional divisor on $Z$ such that $\dim p(E) \ge \dim Y -1$.
\begin{enumerate}
\item\label{lemma:small1}Then $\dim p(E) = \dim Y - 1$. In particular, $E$ is invariant under the foliation on $Z$ induced by $p$.
\item\label{lemma:small2}Moreover, if $z$ is a general point in $E$, then there exists a curve $T \subseteq E$ passing through $z$ with 
$\dim p(T)=1$ such that $q(E_{p(t_1)}(t_1))=q(E_{p(t_2)}(t_2))$ for general points $t_1$ and $t_2$ in $T$, where 
$E_{p(t)}(t)$ denotes the irreducible component of $E_{p(t)} \subseteq p^{-1}(p(t))$ passing through $t \in T \subset E$.
\end{enumerate}
\end{lem}

\begin{proof}For the reader's convenience, the proof is subdivided into a number of steps.

\medskip

\textbf{Step 1. Reduction to the case where $K_\sF$ is Cartier and proof of item 1.}

\medskip

Replacing $X$ by an open neighborhood of the generic point of $q(E)$, we may assume without loss of generality that there exists a positive integer $m$ such that 
$$\sO_X(m{K_\sF})\simeq \sO_X.$$
Let $f \colon X_1 \to X$ be the associated cyclic cover, which is quasi-\'etale (see \cite[Def. 2.52]{KM}), and let $Z_1$ be the normalization of the product $Z \times_{X} X_1$. The induced morphism 
$g\colon Z_1 \to Z$ is then a finite cover. We obtain a commutative diagram as follows:
\begin{center}
\begin{tikzcd}[row sep=large, column sep=large]
Z_1 \ar[dd, bend right, "{p_1}"']\ar[d, "{g}"]\ar[r, "{q_1}"]  & X_1\ar[d, "f"]\\
Z \ar[d, "{p}"]\ar[r, "{q}"]  & X\\
Y. &
\end{tikzcd}
\end{center}
Notice that $X_1$ has klt singularities by \cite[Prop. 3.16]{Kollar97}, and that the foliation $\sF_{X_1}:=f^{-1}\sF$ is weakly regular by Lemma \ref{lemma:regular_quasi_etale}. 
Observe now that the foliation 
$\sF_{Z_1}:=q_1^{-1}\sF_{X_1}$ is given by $p_1$ and that $Z_1$ identifies with the normalization of the graph of the rational map $p_1\circ q_1^{-1}$. By item~\ref{cor:weakly_regular_universal_family_leaves1} in Corollary \ref{cor:weakly_regular_universal_family_leaves}, $\sF_{Z_1}$ is weakly regular. 
Let $E_1$ be a prime divisor on $Z_1$ such that $g(E_1)=E$. Notice that 
$E_1$ is $q_1$-exceptional and that $\dim p(E) = \dim p_1(E_1)$.
Thus, replacing $X$ by $X_1$, we may assume without loss of generality that 
$$K_\sF \sim_\mathbb{Z} 0.$$

\medskip

Then, by item~\ref{cor:weakly_regular_universal_family_leaves2} in Corollary~\ref{cor:weakly_regular_universal_family_leaves}, we must have $p(E)\subsetneq Y$. 
It follows that $p(E)$ is a prime divisor on $Y$ since $\dim p(E) \ge \dim Y -1$ by assumption. 
In particular, $E$ is invariant under the foliation $\sF_Z:=q^{-1}\sF$.

\medskip

\textbf{Step 2. The foliation induced by $\sF$ on $q(E)$.}

\medskip

Set $B:=q(E)$, and let $E^\circ \subseteq E \cap Z_\textup{reg}$ be a non-empty open set. We obtain a commutative diagram as follows: 

\begin{center}
\begin{tikzcd}[row sep=large, column sep=large]
E^\circ \ar[d, hookrightarrow]\ar[r, "{a}"] \ar[dd, bend right, "{j}"'] & B \ar[d, equal]\\
E\ar[d, hookrightarrow]\ar[r, two heads] & B \ar[d, hookrightarrow, "i"]\\
Z \ar[d, "{p}"']\ar[r, "{q}"]  & X\\
Y. &
\end{tikzcd}
\end{center}
Shrinking $X$, if necessary, we may assume without loss of generality that $B$ is smooth.
By \cite[Thm. 1.3 and Prop. 6.1]{kebekus_pull_back}, there is a factorization:

\begin{center}
\begin{tikzcd}[column sep=large]
{\Omega_X^r}|_{B} \ar[r]\ar[rr, bend left, two heads, "{di}"] & {\Omega_X^{[r]}}|_{B} \ar[r, "{d_{\textup{refl}}i}"'] & \Omega_B^r.
\end{tikzcd}
\end{center}
This implies that the map ${\Omega_X^{[r]}}|_{B} \to \Omega_B^r$ is surjective.

\begin{claim}\label{claim:projectable}
The foliation $\sF_{E^\circ}$ on $E^\circ$ induced by $\sF_Z$ is projectable under $a$. 
\end{claim}

\begin{proof}
Let $$\nu_X\colon \Omega_X^{[r]} \twoheadrightarrow \sO_X(K_{\sF}) \quad \textup{and} \quad \nu_Z\colon \Omega_Z^{[r]} \to \sO_Z(K_{\sF_Z})$$ be the Pfaff fields associated to $\sF$ and $\sF_Z$ respectively. Since $E^\circ$ is invariant by $\sF_Z$, there is a factorization:

\begin{center}
\begin{tikzcd}[column sep=large]
{\Omega_Z^{r}}|_{E^\circ} \ar[r]\ar[d, two heads] & {\Omega_Z^{[r]}}|_{E^\circ} \ar[d, "{d_{\textup{refl}}j}"'] \ar[r, "{{\nu_Z}|_{E^\circ}}"] & \sO_Z(K_{\sF_Z})|_{E^\circ} \ar[d, equal]\\
\Omega_{E^\circ}^{r} \ar[r, equal]  & \Omega_{E^\circ}^{r} \ar[r] & \sO_Z(K_{\sF_Z})|_{E^\circ}.
\end{tikzcd}
\end{center}
Recall from the proof of Corollary \ref{cor:weakly_regular_universal_family_leaves1} that there is a commutative diagram:
\begin{center}
\begin{tikzcd}[column sep=large]
q^*\Omega_X^{[r]} \ar[r, two heads, "{q^*\nu_X}"]\ar[d, "{d_{\textup{refl}}q}"'] &  q^*\sO_X(K_\sF) \\
\Omega_{Z}^{[r]} \ar[r, two heads, "{\nu_Z}"]  & \sO_Z(K_{\sF_Z}) \ar[u, "{\sim}" labl].
\end{tikzcd}
\end{center}
Finally, by \cite[Prop. 6.1]{kebekus_pull_back}, the diagram 
\begin{center}
\begin{tikzcd}[column sep=huge]
{(q^*\Omega_X^{[r]})}|_{E^\circ}\simeq a^*({\Omega_X^{[r]}}|_{B}) \ar[r, "{a^*d_{\textup{refl}}i}"]\ar[d, "{d_{\textup{refl}}q|_{E^\circ}}"'] &  a^*\Omega_B^r \ar[d] \\
{\Omega_{Z}^{[r]}}|_{E^\circ} \ar[r, "{d_{\textup{refl}}j}"']  & \Omega_{E^\circ}^r
\end{tikzcd}
\end{center}
is commutative as well. Therefore, we have a commutative diagramm as follows:
\begin{center}
\begin{tikzcd}[column sep=huge]
{(q^*\Omega_X^{[r]})}|_{E^\circ}\simeq a^*({\Omega_X^{[r]}}|_{B}) \ar[r, two heads, "{a^*d_{\textup{refl}}i}"]\ar[dd, two heads, "(q^*\nu_X)|_{E^\circ}"'] &  a^*\Omega_B^r \ar[d] \\
  & \Omega_{E^\circ}^r \ar[d]\\
(q^*\sO_X(K_\sF))|_{E^\circ} & \sO_Z(K_{\sF_Z})|_{E^\circ}\ar[l, "{\sim}"]
\end{tikzcd}
\end{center}
This in turn implies that there is a factorization:
\begin{center}
\begin{tikzcd}[column sep=large]
{\Omega_X^{[r]}}|_{B} \ar[d, two heads, "{{\nu_X}|_{B}}"]\ar[r, two heads, "{d_{\textup{refl}}i}"] & \Omega_B^r \ar[d]\\
{\sO_X(K_\sF)}|_{B}\ar[r, equal] & {\sO_X(K_\sF)}|_{B}
\end{tikzcd}
\end{center}
whose pull-back to $E^\circ$ gives the diagram above. 
It follows that the map $$\Omega_B^r \twoheadrightarrow  {\sO_X(K_\sF)}|_{B}$$ is the Pfaff field associated to a weakly regular foliation $\sF_B$ of rank $r$ on $B$ such that $da({\sF_{E^\circ}})=\sF_B$. This completes the proof of the claim.
\end{proof}

Then item 2 is an immediate consequence of Claim \ref{claim:projectable} above.
\end{proof}

We are now ready to prove Theorem \ref{thm:regular_foliation_morphism}.

\begin{proof}[Proof of Theorem \ref{thm:regular_foliation_morphism}]
Let $p\colon Z \to Y$ be the family of leaves, and let $q\colon Z \to X$ be the natural morphism.
Since $p$ has connected fibers by construction, Lemma \ref{lemma:irreducible_fibers} applied to $p\circ q^{-1}$ implies that $p$ has irreducible fibers over a big open set contained in $Y$. Hence, to prove Theorem \ref{thm:regular_foliation_morphism}, it suffices to show 
that $\textup{Exc}\,q$ is empty.
 
We argue by contradiction and assume that $\textup{Exc}\,q\neq \emptyset$.
Let $E$ be an irreducible component of $\textup{Exc}\,q$. Then
$E$ has codimension $1$ since $X$ is $\mathbb{Q}$-factorial by assumption. 
Recall from Lemma \ref{lemma:irreducible_fibers} that $p^{-1}(y)$ is irreducible for a general point $y$ in $p(E)$. 
Therefore, by Lemma \ref{lemma:small}, we must have $E=p^{-1}(p(E))$.
Moreover, if $y$ is a general point in $p(E)$, then there exists a curve $T \subseteq p(E)$ passing through $y$ such that $q(p^{-1}(t_1))=q(p^{-1}(t_2))$ for general points $t_1$ and $t_2$ in $T$. Now, there exists a positive integer $t$ such that the cycle theoretic fiber $p^{[-1]}(y)$ is $t[p^{-1}(y)]$ for a general point $y$ in 
$p(E)$. It follows that the restriction of the map $Y \to \textup{Chow}(X)$ to 
$p(E)$ has positive dimensional fibers, yielding a contradiction.
This finishes the proof of the theorem.
\end{proof}

\begin{rem}\label{rem:small}
In the setup of Theorem \ref{thm:regular_foliation_morphism}, let $p\colon Z \to Y$ be the family of leaves,  and let $q\colon Z\to X$
be the natural morphism. If $X$ is only assumed to have klt singularities, then the same argument used in the proof of the theorem shows that $q$ is a small birational map. We have $$K_{Z/Y}-R(p) \sim_\mathbb{Q}q^*K_\sF,$$ where $R(p)$ denotes the ramification divisor of $p$. In particular, if $F$ denotes the normalization of the closure of a general leaf of $\sF$, then $${K_\sF}|_{F}\sim_\mathbb{Q}K_F.$$
\end{rem}

\subsection{A splitting theorem} The following is the main result of this section.

\begin{thm}\label{thm:splitting}
Let $X$ be a normal projective variety,
and let $$T_X = \bigoplus_{i\in I} \sF_i$$ be a decomposition of $T_X$ into 
involutive subsheaves with algebraic leaves. Suppose that
there exists a $\mathbb{Q}$-divisor $\Delta$ such that $(X,\Delta)$ is klt. Then there exists 
a quasi-\'etale cover $f\colon Y \to X$ as well as a decomposition 
$$Y \simeq \prod_{i\in I} Y_i$$ of $Y$ into a product of normal projective varieties
such that the decomposition $T_X = \bigoplus_{i\in I} \sF_i$ lifts to the canonical
decomposition $$T_{\prod_{i\in I} Y_i}= \bigoplus_{i \in I} \textup{pr}_i^*T_{Y_i}.$$
\end{thm}

\begin{proof}To prove the theorem, it is obviously enough to consider the case where $I=\{1,2\}$. Set $\tau(i)=3-i$ for each $i\in \{1,2\}$.

\medskip

\textbf{Step 1. Reduction to the case where $X$ is $\mathbb{Q}$-factorial with klt singularities.}

\medskip

Let $\pi\colon Z \to X$ be a $\mathbb{Q}$-factorialization, whose existence is established in \cite[Cor. 1.37]{kollar_kovacs_singularities}. Recall that $\pi$ is a small birational projective morphism and that $Z$ is $\mathbb{Q}$-factorial with klt singularities. Then we have the decomposition $$T_Z = \pi^{-1}\sF_1\oplus \pi^{-1}\sF_2$$
into involutive subsheaves with algebraic leaves. 

Suppose that there exist normal projective varieties $W_1$ and $W_2$ and a quasi-\'etale cover
$$g \colon W_1 \times W_2 \to Z$$ such that the decomposition $T_Z = \pi^{-1}\sF_1\oplus \pi^{-1}\sF_2$ lifts to the canonical
decomposition $$T_{W_1\times W_2}= \textup{pr}_1^*T_{W_1}\oplus \textup{pr}_2^*T_{W_2}.$$ The Stein factorization $$f \colon Y \to X$$ of $\pi\circ g$ is then a quasi-\'etale cover, and the natural map
$$W_1\times W_2 \to Y$$ is a small birational morphism. Moreover, by \cite[Prop. 3.16]{Kollar97}, $Y$ has klt singularities. In particular, it has rational singularities. Lemma \ref{lemma:splitting} below applied to $Y \map W_1\times W_2$
then implies that $X$ satisfies the conclusion of Theorem \ref{thm:splitting}.

\medskip

Therefore, replacing $X$ by $Z$, if necessary, we may assume without loss of generality that $X$ is $\mathbb{Q}$-factorial with klt singularities.

\medskip

\textbf{Step 2. Covering construction.}

\medskip

By Lemma \ref{lemma:direct_summand_regular}, $\sF_i$ is a weakly regular foliation. Therefore, by Theorem \ref{thm:regular_foliation_morphism}, $\sF_i$ is induced by a surjective equidimensional morphism $p_i\colon X \to T_i$ onto a normal projective variety $T_i$. Moreover, $p_i$ has irreducible fibers over a big open set contained in $T_i$. Let $F_i$ be a general fiber of $p_{\tau(i)}$.

Let $M_i$ denote the normalization of the product $F_i \times_{T_i} X$, and let $M_i \to N_i \to X$ denote the Stein factorization of the natural morphism $M_i \to X$. We will show that $N_i \to X$ is a quasi-\'etale cover.
Notice that for any prime $P$ on $T_i$, $p_i^*P$ is well-defined (see \cite[Sec. 2.7]{cd1fzerocan}) and has irreducible support. 

Write $p_i^*P=m Q$ for some prime divisor $Q$ on $X$ and some integer $m \ge 1$.
Set $n:=\dim X$, and $s:=\dim T_i$. By general choice of $F_i$, we may assume that $F_i \setminus X_{\rm reg}$ has codimension at least $2$ in $F_i$. In particular, $F_i\cap Q \cap X_{\rm reg}\neq\emptyset$. Let $x \in F_i\cap Q \cap X_{\rm reg}$ be a general point.
Since $\sF_1$ and $\sF_2$ are regular foliations at $x$ and $T_X=\sF_1\oplus\sF_2$, there exist local analytic coordinates centered at $x$ and $p_i(x)$ 
respectively such that $p_i$ is given by $$(x_1,x_2,\ldots,x_n)\mapsto (x_1^m,x_2\ldots,x_s),$$ and such that 
$F_i$ is given by the equations $$x_{s+1}=\cdots=x_n=0.$$
A straightforward local computation then shows that $N_i \to X$ is a quasi-\'etale cover over the generic point of 
$p_i^{-1}(P)$. This immediately implies that $N_i \to X$ is a quasi-\'etale cover.

Let $Y$ be the normalization of $X$ in the compositum of the function fields $\mathbb{C}(N_i)$, and 
let $f\colon Y \to X$ be the natural morphism. Set $\sG_i:=f^{-1}\sF_i$. By construction, $f$ is a quasi-\'etale cover, and 
$\sG_i$ is induced by a surjective equidimensional morphism $q_i\colon Y \to R_i$ with reduced fibers over a big open set contained in $R_i$. Moreover, there exists a subvariety $G_i \subseteq f^{-1}(F_i)$
such that the restriction $G_i \to R_i$ of $q_i$ to $G_i$ is 
a birational morphism.  

\medskip

\textbf{Step 3. End of proof.}

\medskip

Let $R_i^\circ$ denote the smooth locus of $R_i$, and set $Y_i^\circ:=q_i^{-1}(R_i^\circ)$.
Let $Z_i^\circ \subseteq Y_i^\circ$ be the open set where 
${q_i}|_{Y_i^\circ}$ is smooth. Notice that $Z_i^\circ$ has complement of codimension at least $2$
in $Y_i^\circ$ since $q_i$ has reduced fibers over a big open set contained in $R_i$. 

The restriction of the tangent map 
$$T{q_i}|_{Y_i^\circ}\colon T_{Y_i^\circ}\to \big({q_i}|_{Y_i^\circ}\big)^*T_{R_i^\circ}$$
to ${\sG_{\tau(i)}}|_{Z_i^\circ} \subseteq T_{Z_i^\circ}$ 
then induces an isomorphism ${\sG_{\tau(i)}}|_{Z_i^\circ}\simeq \big({q_i}|_{Z_i^\circ}\big)^*T_{R_i^\circ}$. Since 
${\sG_{\tau(i)}}|_{Y_i^\circ}$ and $\big({q_i}|_{Y_i^\circ}\big)^*T_{R_i^\circ}$
are both reflexive sheaves, we finally obtain an isomorphism
$${\sG_{\tau(i)}}|_{Y_i^\circ}\simeq \big({q_i}|_{Y_i^\circ}\big)^*T_{R_i^\circ}.$$
A classical result of complex analysis says that complex flows of vector fields on analytic spaces exist (see \cite{kaup}).  It follows that
${q_i}|_{Y_i^\circ}$ is a locally trivial analytic fibration for the analytic topology. 

The morphism $q_1\times q_2\colon Y \to R_1 \times R_2$ then induces an isomorphism
$$q_1^{-1}(R_1^\circ)\cap q_2^{-1}(R_2^\circ) \simeq R_1^\circ \times R_2^\circ$$ 
since $G_1 \cdot G_2 =1$ and $q_i$ is locally trivial over $R_i^\circ$.
In particular, $q_1\times q_2$ is a small birational morphism. By \cite[Prop. 3.16]{Kollar97} again, $Y$ has klt singularities. Hence, it has rational singularities. Lemma \ref{lemma:splitting} below applied to $q_1\times q_1$ then
implies that $X$ satisfies the conclusion of Theorem \ref{thm:splitting}, completing 
the proof of the theorem.
\end{proof}

\begin{lem}[{\cite[Prop. 18]{kollar_larsen}}]\label{lemma:splitting}
Let $X$, $Y_1$ and $Y_2$ be normal projective varieties, and let $\pi\colon X \map Y_1\times Y_2$ be a birational map that 
does not contract any divisor. Suppose in addition that $X$ has rational singularities. Then 
$X$ decomposes as a product $X\simeq X_1 \times X_2$ and there exist birational maps $\pi_i\colon X_i \to Y_i$ such that 
$\pi = \pi_1 \times \pi_2$.
\end{lem}

\section{Proof of Theorem \ref{thm:main_thm}}

The present section is devoted to the proof of Theorem \ref{thm:main_thm}.

\begin{proof}[Proof of Theorem \ref{thm:main_thm}]

We have seen in Theorem~\ref{polystable} that the tangent sheaf of $X$ is polystable. By definition it means that we have a decomposition
$$T_X= \bigoplus_{i\in I} \sF_i$$
where the $\sF_i$ are stable with respect to $c_1(X)$ and have the same slope. Moreover, each subsheaf $\sF_i$ defines on $X_{\rm reg}$ a parallel subbundle of $T_{X_{\rm reg}}$ with respect to the Kähler-Einstein metric ${\omke}|_{X_{\rm reg}}$. 
This immediately implies that ${\sF_i}|_{X_{\rm reg}}$ is involutive.

\begin{claim}
Each foliation $\sF_i$ has algebraic leaves. 
\end{claim}

\begin{proof}
Let $m$ be a positive integer such that $-mK_X$ is very ample, and let $C \subset X$ be a general complete intersection curve of elements 
in $|-mK_X|$. By general choice of $C$, we may assume that $C \subset X_{\rm reg}$ and that $\sF_i$ is locally free in a neighborhood of $C$.
If $m$ is large enough, then the vector bundle ${\sF_i}|_{C}$ is semistable by \cite[Thm.1.2]{Flenner84}). We conclude that it is 
ample since it has positive slope. Then \cite[Fact 2.1.1]{bogomolov_mcquillan01} says that $\sF_i$ has algebraic leaves.
Alternatively, one can apply \cite[Thm. 1.1]{CP19} to the foliation $\wh{\sF_i}$ on $\X$ induced by $\sF_i$.
\end{proof}

Let $f:Y\to X$ be the quasi-étale cover and $Y=\prod_{i \in I} Y_i$ be the splitting that are both provided by Theorem~\ref{thm:splitting}. The decomposition 
\begin{equation}
\label{TY}
T_Y=\bigoplus_{i\in I} \mathrm{pr}_i^* T_{Y_i}
\end{equation} 
is a decomposition of $T_Y$  into summands of maximal slope. If there exists $i\in I$ such that  $T_{Y_i}$ is not stable with respect to $c_1(Y_i)$, then it means that the polystable decomposition of $T_Y$ provided by Theorem~\ref{polystable} via $f^*\omke$ refines strictly the decomposition \eqref{TY}. By applying Theorem~\ref{thm:splitting} again, we can find another quasi-étale cover $Y'\to Y$ which splits according to the polystable decomposition of $T_Y$ and one can then compare again the polystable decomposition of $T_{Y'}$ to the one coming from $T_Y$. After finitely many such steps, one can find a quasi-étale cover $g:Z\to X$ such that
\begin{enumerate}
\item[$(i)$] There exists a splitting $Z=\prod_{k\in K} Z_k$ into a product of $\mathbb Q$-Fano varieties.
\item[$(ii)$] For any $k\in K$, the tangent sheaf $T_{Z_k}$ is stable with respect to $c_1(Z_k)$. 
\item[$(iii)$] The variety $Z$ admits a Kähler-Einstein metric given by $g^*\omke$. 
\end{enumerate}

\noindent
Theorem~\ref{thm:main_thm} is a consequence of the Claim below. 

\begin{claim}
\label{produit}
There exist a Kähler-Einstein metric $\om_k$ on each variety $Z_k$ such that $g^*\om=\sum_{k\in K} \mathrm{pr}_k^*\om_k$. 
\end{claim}

\begin{proof}[Proof of Claim~\ref{produit}]
We set $n_k:=\dim Z_k$. As the subsheaf $\sF_k:=\mathrm{pr}_k^*T_{Z_k} \subset T_Z$ has maximal slope with respect to $c_1(Z)$, it follows from Theorem~\ref{polystable} that $\sF_k|_{Z_{\rm reg}}$ is parallel with respect to $g^*\omke$. This enables us to define a smooth Kähler metric $\om_k$ on $Z_k^{\rm reg}$ such that $g^*\omke = \sum_{k\in K} \mathrm{pr_k}^*\om_k$ on $Z_{\rm reg}$. Clearly, one has $\Ric \om_k= \om_k$ on $Z_k^{\rm reg}$. In order to check that $\om_k$ defines a Kähler-Einstein metric on $Z_k$ in the sense of Definition~\ref{defKE}, it is sufficient to check that $\int_{Z_k^{\rm reg}} \om_k^{n_k}=c_1(Z_k)^{n_k}$ by Remark~\ref{criterion}. By \cite[Proposition~3.8]{BBEGZ} we always have the inequality $\int_{Z_k^{\rm reg}} \om_k^{n_k}\le c_1(Z_k)^{n_k}$ and therefore
\begin{align*}
c_1(Z)^n&=\int_{Z_{\rm reg}} g^*\omke^n \\
&= \prod_{k\in K} \int_{Z_k^{\rm reg}} \om_k^{n_k} \\
& \le \prod_{k\in K}  c_1(Z_k)^{n_k}.
\end{align*}
Since $c_1(Z)^n= \prod_{k\in K}  c_1(Z_k)^{n_k}$, one must have $\int_{Z_k^{\rm reg}} \om_k^{n_k}=c_1(Z_k)^{n_k}$ for all $k\in K$. 
\end{proof}

\noindent
Theorem~\ref{thm:main_thm} is now proved. 
\end{proof}

\bibliographystyle{skalpha}
\bibliography{biblio}

\newcommand{\etalchar}[1]{$^{#1}$}
\providecommand{\bysame}{\leavevmode\hbox to3em{\hrulefill}\thinspace}
\providecommand{\MR}{\relax\ifhmode\unskip\space\fi MR}
\providecommand{\MRhref}[2]{%
  \href{http://www.ams.org/mathscinet-getitem?mr=#1}{#2}
}
\providecommand{\href}[2]{#2}
\begin{thebibliography}{GKKP11}

\bibitem[ADK08]{adk08}
{\sc C.~Araujo, S.~Druel, and S.~J. Kov{\'a}cs}: \emph{Cohomological
  characterizations of projective spaces and hyperquadrics}, Invent. Math.
  \textbf{174} (2008), no.~2, 233--253.

\bibitem[BBJ15]{BBJ}
{\sc R.~Berman, S.~Boucksom, and M.~Jonsson}: \emph{{A variational approach to
  the Yau-Tian-Donaldson conjecture}}, Preprint
  \href{http://arxiv.org/abs/1509.04561}{arXiv:1509.04561}, 2015.

\bibitem[BBE{\etalchar{+}}19]{BBEGZ}
{\sc R.~J. Berman, S.~Boucksom, P.~Eyssidieux, V.~Guedj, and A.~Zeriahi}:
  \emph{K\"{a}hler-{E}instein metrics and the {K}\"{a}hler-{R}icci flow on log
  {F}ano varieties}, J. Reine Angew. Math. \textbf{751} (2019), 27--89.
  {\sf\scriptsize 3956691}

\bibitem[BG14]{BG}
{\sc R.~J. Berman and H.~Guenancia}: \emph{{K{\"a}hler-Einstein metrics on
  stable varieties and log canonical pairs}}, Geometric and Function Analysis
  \textbf{24} (2014), no.~6, 1683--1730.

\bibitem[BM16]{bogomolov_mcquillan01}
{\sc F.~Bogomolov and M.~McQuillan}: \emph{Rational curves on foliated
  varieties}, Foliation theory in algebraic geometry, Simons Symp., Springer,
  Cham, 2016, pp.~21--51.

\bibitem[Bra20]{Braun20}
{\sc L.~Braun}: \emph{{The local fundamental group of a Kawamata log terminal
  singularity is finite}}, Preprint
  \href{http://arxiv.org/abs/2004.00522}{arXiv:2004.00522}, 2020.

\bibitem[CP19]{CP19}
{\sc F.~Campana and M.~P\u{a}un}: \emph{Foliations with positive slopes and
  birational stability of orbifold cotangent bundles}, Publ. Math. Inst. Hautes
  \'{E}tudes Sci. \textbf{129} (2019), 1--49. {\sf\scriptsize 3949026}

\bibitem[CDS15a]{CDS1}
{\sc X.~Chen, S.~Donaldson, and S.~Sun}: \emph{K\"ahler-{E}instein metrics on
  {F}ano manifolds. {I}: {A}pproximation of metrics with cone singularities},
  J. Amer. Math. Soc. \textbf{28} (2015), no.~1, 183--197. {\sf\scriptsize
  3264766}

\bibitem[CDS15b]{CDS2}
{\sc X.~Chen, S.~Donaldson, and S.~Sun}: \emph{K\"ahler-{E}instein metrics on
  {F}ano manifolds. {II}: {L}imits with cone angle less than {$2\pi$}}, J.
  Amer. Math. Soc. \textbf{28} (2015), no.~1, 199--234. {\sf\scriptsize
  3264767}

\bibitem[CDS15c]{CDS3}
{\sc X.~Chen, S.~Donaldson, and S.~Sun}: \emph{K\"ahler-{E}instein metrics on
  {F}ano manifolds. {III}: {L}imits as cone angle approaches {$2\pi$} and
  completion of the main proof}, J. Amer. Math. Soc. \textbf{28} (2015), no.~1,
  235--278. {\sf\scriptsize 3264768}

\bibitem[Dem12]{Dem1}
{\sc J.-P. Demailly}: \emph{Complex {A}nalytic and {D}ifferential {G}eometry},
  OpenContent Book, freely available from the author's web site., September
  2012.

\bibitem[DNL17]{LDN}
{\sc E.~Di~Nezza and C.~H. Lu}: \emph{Complex {M}onge-{A}mp\`ere equations on
  quasi-projective varieties}, J. Reine Angew. Math. \textbf{727} (2017),
  145--167. {\sf\scriptsize 3652249}

\bibitem[Dru17]{druel15}
{\sc S.~Druel}: \emph{On foliations with nef anti-canonical bundle}, Trans.
  Amer. Math. Soc. \textbf{369} (2017), no.~11, 7765--7787.

\bibitem[Dru18]{Dru16}
{\sc S.~Druel}: \emph{A decomposition theorem for singular spaces with trivial
  canonical class of dimension at most five}, Invent. Math. \textbf{211}
  (2018), no.~1, 245--296. {\sf\scriptsize 3742759}

\bibitem[Dru20]{cd1fzerocan}
{\sc S.~Druel}: \emph{Codimension one foliations with numerically trivial
  canonical class on singular spaces}, Duke Math. J., to appear, 2020.

\bibitem[Eno88]{Enoki}
{\sc I.~Enoki}: \emph{Stability and negativity for tangent sheaves of minimal
  {K}{\"a}hler spaces}, Geometry and analysis on manifolds ({K}atata/{K}yoto,
  1987), Lecture Notes in Math., vol. 1339, Springer, Berlin, 1988,
  pp.~118--126. {\sf\scriptsize 961477 (90a:32039)}

\bibitem[EGZ09]{EGZ}
{\sc P.~Eyssidieux, V.~Guedj, and A.~Zeriahi}: \emph{{Singular
  K{\"a}hler-Einstein metrics}}, {J. Amer. Math. Soc.} \textbf{22} (2009),
  607--639.

\bibitem[Fle84]{Flenner84}
{\sc H.~Flenner}: \emph{Restrictions of semistable bundles on projective
  varieties}, Comment. Math. Helv. \textbf{59} (1984), no.~4, 635--650,
  \href{http://dx.doi.org/10.1007/BF02566370}{DOI: 10.1007/BF02566370}.
  {\sf\scriptsize MR780080 (86m:14014)}

\bibitem[GGK19]{GGK}
{\sc D.~Greb, H.~Guenancia, and S.~Kebekus}: \emph{Klt varieties with trivial
  canonical class: holonomy, differential forms, and fundamental groups}, Geom.
  Topol. \textbf{23} (2019), 2051–2124.

\bibitem[GKKP11]{GKKP}
{\sc D.~Greb, S.~Kebekus, S.~J. Kov{\'a}cs, and T.~Peternell}:
  \emph{Differential forms on log canonical spaces}, Publ. Math. Inst. Hautes
  {\'E}tudes Sci. (2011), no.~114, 87--169. {\sf\scriptsize 2854859}

\bibitem[GKP20]{GKP20}
{\sc D.~Greb, S.~Kebekus, and T.~Peternell}: \emph{{ Projective flatness over
  klt spaces and uniformisation of varieties with nef anti-canonical divisor}},
  Preprint \href{http://arxiv.org/abs/2006.08769}{arXiv:2006.08769}, 2020.

\bibitem[GZ12]{GZ11}
{\sc V.~Guedj and A.~Zeriahi}: \emph{Stability of solutions to complex
  {M}onge-{A}mp\`ere equations in big cohomology classes}, Math. Res. Lett.
  \textbf{19} (2012), no.~5, 1025--1042. {\sf\scriptsize 3039828}

\bibitem[Gue16]{GSS}
{\sc H.~Guenancia}: \emph{{Semistability of the tangent sheaf of singular
  varieties}}, Algebraic Geometry \textbf{3} (2016), no.~5, 508--542.

\bibitem[HP19]{HP}
{\sc A.~H\"{o}ring and T.~Peternell}: \emph{Algebraic integrability of
  foliations with numerically trivial canonical bundle}, Invent. Math.
  \textbf{216} (2019), no.~2, 395--419. {\sf\scriptsize 3953506}

\bibitem[Kau65]{kaup}
{\sc W.~Kaup}: \emph{Infinitesimale {T}ransformationsgruppen komplexer
  {R}\"aume}, Math. Ann. \textbf{160} (1965), 72--92.

\bibitem[Keb13]{kebekus_pull_back}
{\sc S.~Kebekus}: \emph{Pull-back morphisms for reflexive differential forms},
  Adv. Math. \textbf{245} (2013), 78--112.

\bibitem[Kob87]{Koba}
{\sc S.~Kobayashi}: \emph{{Differential geometry of complex vector bundles.}},
  {Princeton, NJ: Princeton University Press; Tokyo: Iwanami Shoten
  Publishers}, 1987 (English).

\bibitem[Kol97]{Kollar97}
{\sc J.~Koll{\'a}r}: \emph{Singularities of pairs}, Algebraic
  geometry---{S}anta {C}ruz 1995, Proc. Sympos. Pure Math., vol.~62, Amer.
  Math. Soc., Providence, RI, 1997, pp.~221--287. {\sf\scriptsize 1492525}

\bibitem[Kol13]{kollar_kovacs_singularities}
{\sc J.~Koll\'ar}: \emph{Singularities of the minimal model program}, Cambridge
  Tracts in Mathematics, vol. 200, Cambridge University Press, Cambridge, 2013,
  With a collaboration of S\'andor Kov\'acs.

\bibitem[KL09]{kollar_larsen}
{\sc J.~Koll\'{a}r and M.~Larsen}: \emph{Quotients of {C}alabi-{Y}au
  varieties}, Algebra, arithmetic, and geometry: in honor of {Y}u. {I}.
  {M}anin. {V}ol. {II}, Progr. Math., vol. 270, Birkh\"{a}user Boston, Boston,
  MA, 2009, pp.~179--211.

\bibitem[KM98]{KM}
{\sc J.~Koll{\'a}r and S.~Mori}: \emph{Birational geometry of algebraic
  varieties}, Cambridge Tracts in Mathematics, vol. 134, Cambridge University
  Press, Cambridge, 1998, With the collaboration of C. H. Clemens and A. Corti,
  Translated from the 1998 Japanese original. {\sf\scriptsize 1658959
  (2000b:14018)}

\bibitem[Li17]{ChiLi17}
{\sc C.~Li}: \emph{Yau-{T}ian-{D}onaldson correspondence for {K}-semistable
  {F}ano manifolds}, J. Reine Angew. Math. \textbf{733} (2017), 55--85.
  {\sf\scriptsize 3731324}

\bibitem[Li18]{ChiLi18}
{\sc C.~Li}: \emph{{On the stability of extensions of tangent sheaves on
  Kähler-Einstein Fano/Calabi-Yau pairs}}, Preprint
  \href{http://arxiv.org/abs/1803.01734}{arXiv:1803.01734}, to appear in Math.
  Ann., 2018.

\bibitem[LTW19]{LTW}
{\sc C.~Li, G.~Tian, and F.~Wang}: \emph{{The uniform version of
  Yau-Tian-Donaldson conjecture for singular Fano varieties}}, Preprint
  \href{http://arxiv.org/abs/1903.01215}{arXiv:1903.01215}, 2019.

\bibitem[Tia92]{tian92}
{\sc G.~Tian}: \emph{On stability of the tangent bundles of {F}ano varieties},
  Internat. J. Math. \textbf{3} (1992), no.~3, 401--413. {\sf\scriptsize
  1163733}

\bibitem[Tia15]{T}
{\sc G.~Tian}: \emph{K-stability and {K}\"ahler-{E}instein metrics}, Comm. Pure
  Appl. Math. \textbf{68} (2015), no.~7, 1085--1156. {\sf\scriptsize 3352459}

\bibitem[Yau78]{Yau78}
{\sc S.-T. Yau}: \emph{{On the Ricci curvature of a compact K{\"a}hler manifold
  and the complex Monge-Amp{\`e}re equation. I.}}, Commun. Pure Appl. Math.
  \textbf{31} (1978), 339--411.

\end{thebibliography}

\end{document}